\documentclass[12pt,a4paper,twoside,final,notitlepage, leqno]{article}
\usepackage[english]{babel}
\usepackage[T1]{fontenc}  
\usepackage{graphicx}
\usepackage{amsmath,amsthm,epsfig,amsfonts,bbm}  

\newcommand{\Frac}[2] {\frac{\mbox{\normalsize{$#1$}}}{\mbox{\normalsize{$#2$}}}}

\newcommand{\dt}[2]{\displaystyle \frac{{\rm d} #1}{ {\rm d}  #2}}

\newcommand{\pequationdeb}{$$ \left\{ \begin{minipage}[c]{130mm}}
\newcommand{\pequationfin}{\end{minipage}
                           \right. $$}
\newcommand{\JVT}{\displaystyle J \left( \frac{V(t)}{T(t)} \right)}

\def \smb {{\scriptstyle \bullet }}
\newcommand{\monitem}{ \smallskip \noindent $\bullet$ \quad  } 
\newcommand{\sonitem}{ \smallskip \noindent \quad  $+$ \quad }
\newcommand{\moneq}{\vspace*{-6pt} \begin{equation} \displaystyle } 
\newcommand{\moneqstar}{\vspace*{-6pt} \begin{equation*} \displaystyle } 
\newcommand{\monendstar}{\vspace*{-6pt} \end{equation*}   }
\newcommand{\monend}{\vspace*{-6pt} \end{equation}   }
\newcommand{\beq}     {\begin{equation}}
\newcommand{\enq}     {\end{equation}}
\newcommand{\be}    {\begin{enumerate}}
\newcommand{\ee}    {\end{enumerate}}

\newcommand{\Bb}

\newcommand{\dd}{{\rm d}}

\def\br {\break}

\def\section*#1{}

\def\resume{\if@twocolumn
\section*{R\'esum\'e}
\else \small
\quotation{\bf \it R\'esum\'e \rule[1mm]{1.5mm}{0.2mm}\vspace{0pt}}
\fi}
\def\endresume{\if@twocolumn\else\endquotation\fi}
\def\abstract{\if@twocolumn
\noindent\section*{{\bf Abstract}}
\else \small
\quotation{\noindent \bf {Abstract.} \rule[1mm]{1.5mm}{0.2mm}\vspace{0pt}}
\fi}
\def\endabstract{\if@twocolumn\else\endquotation\fi}

\newtheorem{theorem}{Theorem}
\newtheorem{lemma}[theorem]{Lemma}
\newtheorem{proposition}[theorem]{Proposition}
\newtheorem{definition}[theorem]{Definition}

\begin{document} 
\title{\bf \LARGE   
Mathematical modeling of antigenicity    ~\\ ~ \vspace{-.8cm}  ~\\ 
  for HIV dynamics  ~\\~    }  

\author { { \large  Fran\c{c}ois Dubois~$^{ab}$, 
Herv\'e V.J. Le Meur~$^{b}$ and Claude Reiss~$^{c}$}     \\ ~\\
{\it \small  $^a$   Conservatoire National des Arts et M\'etiers, }  \\
{\it  \small  Department of Mathematics,  Paris, France.} \\  
{\it  \small $^b$   Department of Mathematics, University  Paris-Sud,} \\
{\it \small B\^at. 425, F-91405 Orsay Cedex, France.} \\ 
{\it  \small $^c$  Vigilent Technologies, 38160 Chevri{\`e}res, France.}  \\
{\it \small francois.dubois@math.u-psud.fr, herve.lemeur@math.u-psud.fr, 
claude.reiss@vigilentech.com  }    \\   ~\\    }  
\date{ 5  March   2010~\protect\footnote{Published in  {\it Maths In Action}, 
 volume 3, number 1, (2010) pages 1-35. Edition pr\'esente du 03 January 2011.}}
  
\maketitle

\begin{abstract} 
This contribution is devoted to a new model of HIV multiplication
motivated by the patent of one of the authors. We take into account
the antigenic diversity through what we define ``antigenicity'',
whether of the virus or of the adapted lymphocytes. We model
the interaction of the immune system and the viral
strains by two processes. On the one hand, the presence of a
given viral quasi-species generates antigenically adapted
lymphocytes. On the other hand, the lymphocytes kill only viruses for which
they have been designed. We consider also the mutation and
multiplication of the virus. An original infection term is derived.

\noindent
So as to compare our system of differential equations with
well-known models, we study some of them and compare their predictions
to ours in the reduced case of only one antigenicity. In this
particular case, our model does not yield any major qualitative
difference. We prove mathematically that, in this case, our model is
biologically consistent (positive fields) and has a unique continuous
solution for long time evolution. In conclusion, this model improves
the ability to simulate more advanced phases of the disease.             
 $ $ \\[5mm]
\end{abstract}

\bigskip \bigskip  \newpage \noindent {\bf \Large 1) \quad  Introduction}  
 
\smallskip  
Virus multiplication is at the basis of viral infection. Although the viral replication
cycle involved makes heavy use of the infected cell's resources \cite{Fujii,Warrilow}, the
enzyme(s) in charge of viral genome replication is(are) frequently encoded in the latter
\cite{Brass}. Compared to cellular polymerases, viral polymerases are usually more
error-prone \cite{Kashkina}. It follows that the viral mutation rate, defined as the
average number of base changes at a given position of the genome per replication cycle,
may be large compared to that observed in our own cells, which is about one in a billion
\cite{Petravic}. For certain viruses, in particular retroviruses like HIV-1, it could be
up to 1,000,000 times greater. By the end of 2007, the Los Alamos HIV data base listed
over 230,000 different viral sequences (see www.hiv.lanl.gov). Given the small genome size
of these viruses, chances are then that each member of a viral progeny carries
mutations. A single point mutation may be enough to simultaneously affect two genes
encoded in different, but overlapping reading frames, whilst silent mutations, which do
not change amino acid coding, may nevertheless have important biological consequences
\cite{Mueller}.

A virus is mainly characterized by its ability to infect target cells (infectivity) and by
its antigenic signature (antigenicity), defined as both the capacity to induce an immune
response and also its strength and type. Immunogenicity is the ability of antigens to
elicit a response from cells of the immune system. Mutations during virus replication may
therefore release infective or non-infective viruses, of the same or of different
antigenicity. For HIV-1, the ratio of infectious to non-infectious particles is estimated
to range from 1:1 to 1:60,000, depending on the type of cell infected and the viral strain
\cite{Thomas}. Whether the virus is infective or not, over 800 mutations affecting HIV-1
antigenicity were identified in its envelop gene (env) alone \cite{Kothe}.

By encoding its own replication enzymes, the virus has control over its replication
fidelity and thereby challenges heavily the immune system, due to the huge burden imposed
by the number of infective virions produced and their antigenic diversity. This burden is
even worse when the virus targets part of the immune system (CD4 displaying cells), as is
the case for HIV-1. In addition, the immune cell proliferation induced by the viral attack
will provide HIV-1 virions with new targets, engaging the cell-virus dynamics in an
exponentially soaring extension regime.

Kinetic modeling is therefore of high interest for understanding the course of
infection. It is a prerequisite for designing and optimizing treatment strategies based on
antiviral drugs. A large number of deterministic and stochastic inter- and intra-cellular
models of HIV dynamics have already been proposed
(\cite{Finzi,Nowak-May,Pastore_Zubelli,perelson02,perelson99mathematical,Snedecor03,Tuckwell}
among others), but none enable the prediction of the course of the disease in all its
phases. For instance, even if, following antiviral treatment, the plasma load of the virus
becomes undetectable, unscheduled bursts occur, probably fed by viral sanctuaries
disseminated in various tissues and organs (lymph or neuronal tissues, gastro-intestinal
or uro-genital tracts etc. See \cite{Recher}). One may assume that, despite
tissue-specific kinetic diversities, the course of infection in all sanctuaries (including
plasma) obeys a common, complex host-predator relation, differing only by
sanctuary-specific parameters. The course of the global infection would then be the result
of all local processes. This result would however not be a simple addition or
superposition, as it is likely that each local process would provide viruses having
locally-specified antigenicities and infectivities which may challenge the immune cells in
the same or other sanctuaries.

What is the infection phenomenology ? As far as our study is concerned, the process
involves four major participants or ``fields'': uninfected T lymphocytes (denoted $T$),
infected ones ($U$), infectious viruses ($V$) and non-infectious viruses ($W$). For each
participant, a characteristic antigenicity is recognized by a lymphocyte, or displayed by
a virus. In the following, we call {\it antigenicity} the variable associated to this
biomolecular characteristic and denote it by the index $j$. Modeling this antigenicity can
also be found in
\cite{Althaus_Boer,Frid_Jabin_Perthame,Nowak_Bangham,Nowak-May,Pastore_Zubelli}. Such
models are {\em microscopic} since they take into account a microscopic quantity that may
not be easily measured biologicaly. Macroscopic models only use the sum of all the $T_j$
and $V_j$. This modeling is necessary, since viruses devoid of antigenic diversity would
be eliminated by the immune system. Actually, viral clearance is not observed and
furthermore, it is the immune system that will ultimately collapse. So as to quantify this
{\it biological reality}, we extend the classical biological definition by assuming we
characterize the biomolecular viruses. Mainly our antigenicity is not linked to any
virulence since we consider only infectious or non-infectious viruses and no intermediate
state.

The evolution with time of $T_j$ depends on three phenomena: regression of the $T_j$
population due to the viral attack (whatever the antigenic pedigree of the infecting
virus), which transforms $T_j$'s into $U_j$'s; the stimulation of the immune system by
both infectious and noninfectious viruses of antigenicity $j$; the natural fate of $T_j$
independently of the viral presence, {\it i.e.} spontaneous generation and death of $T_j$
species. The $T_j$ population may include cross-reacting species, which are active also
against a (limited) number of targets with different antigenicities. The lymphocytes $U_j$
are derived only from the $T_j$ population following viral attack. It is assumed also,
that the switch from the $T_j$ to the $U_j$ state occurs only upon viral infection. Like
the $T_j$, the $U_j$ are subject to natural death. Viruses $V_j$ and $W_j$ are generated
through infection of any susceptible T cell, whatever its antigenicity. The parent virus
of $V_j$ and $W_j$ may be a $V_j$ (no viral antigenicity modification) or a $V_k$ of
different antigenicity (viral antigenicity mutation). Viruses $V_j$ and $W_j$ will be the
target of lymphocytes $T_j$ exclusively. Both viral species have a natural death
rate. Viruses $V_j$ and $W_j$ differ by mutations in genes involved in infectivity, but
not affecting antigenicity~$j$.

This modelling assumes that viral genome parts responsible for infectivity may differ from
those responsible for antigenicity. Since the viral strategy is to escape the immune
response whilst minimizing loss of infectivity, excess mutations in the
antigenicity-specifying part of the viral genome would be more favorable. The immune
system senses mainly the viral surface, hence mutations in the viral envelop genes would
be most beneficial, but they should not, or marginally only, affect viral genome parts
responsible for infectivity. HIV-1 handles in part this dilemma by introducing mutational
hot-spots in its genome (\cite{Ji} and the website {http://www.hiv.lanl.gov}), mainly
in the envelop genes, where the mutation rate is much higher due to local sequence and
structural particularities of the genome \cite{Bebenek,Kothe}.

What is the therapeutic motivation ?

Approved drugs for AIDS treatment are of three kinds mainly. Two of them inhibit reverse
transcription (using nucleotide and non-nucleotide analogs) and the third inhibits a viral
enzyme in charge of cleaving reverse transcriptase from a precursor protein. Because of
its high mutation rate, HIV rapidly develops resistance to any one of these drugs taken
individually. Resistance can be considerably delayed by using various combinations of
these drugs (multitherapies). So far however, no combination has been found that could
clear the virus. Therefore therapies are life-long and unfortunately have considerable
side-effects.

Obviously, the high mutation rate of reverse transcriptase is central to the successful
viral strategy. It allows the virus to escape the circulating immune cells (antigenic
mutations) and to develop drug resistance, although over 90\% of its progeny lacks
infectivity and will therefore be rapidly cleared. The natural viral mutation rate is at
the limit of the ``error threshold'' \cite{Biebricher}, as a slight increase would produce
100\% non-infectious viruses. Conversely, reducing the mutation rate would reduce the
antigenic diversity and allow the immune system to eliminate the stabilized viral strains,
and drug resistance would vanish.

New AIDS therapy at stake?

A promising therapeutic approach would then be to take control of the
viral mutation rate. This was shown to be feasible in \cite{Derrien} with
a CNRS-filed patent based on this work USPTO 6,727,059, but also in
\cite{Drosopoulos,Harris,Murakami}, by supplying the reverse transcriptase with nucleotide
analogs. Some of them relax while others reinforce the replication
fidelity, without blocking reverse transcription. For a review see \cite{Anderson}.

Both therapeutic strategies would give rise to specific viral dynamics. These need to be
understood and assessed in detail. The medical decision to choose one particular strategy
and setting up the adapted drug regimen for a patient, given his viral load and lymphocyte
count (dose and extent of treatment, time expected to reach viral clearance etc) needs
careful analysis. Simulation of viral dynamics with a drug regimen could help in reaching
this decision. To this end, the present work is to build a realistic mathematical model of
viral dynamics.

In the following we review some well-known models by giving an analysis of the stationary
solutions (fixed points) and their stability in Section 2. In Section 3, we derive our new
model precisely and Section 4 is devoted to a full study of its mathematical properties in
a reduced case (only one antigenicity). We discuss this new model in Section 5 and
conclude in Section 6.

\bigskip \bigskip  \newpage \noindent {\bf \Large 2) \quad  Some popular models}

\smallskip 
Throughout this section, we review some well-known models. Some of
them take specific biological reality into account. So as to come to a
common description with our model presented later, we reduced them in
a preliminary step when needed. When used, the fields $T, U, V, W$
have the same meaning as above. When the model has a term identical to
ours, we denote the parameter of the term as ours. When it is
different, we denote it the same way as the authors and add a
subscript depending on the authors. Also, we use the very same values
of common parameters to have comparable results. We take most values
from Snedecor \cite{Snedecor03} and check these values with other
articles (\cite{Herz,Nelson,Perelson},~...).

All the fields of the models are non-dimensionnalized by using the
value of the non-infected lymphocytes at health (no virus and long
time) as a characteristic value both for the lymphocytes (infected or
not) and the virus.

For every model, we look for fixed points and study their
stability. A fixed point of the model is supposed to represent a
biological state lasting. In most articles, a fixed point is even
considered as the state during the second phase where the viremia
increases slowly but is not constant yet. The time scales are not made
precise and so it is not so contradictory.

Throughout the article, we call health the state where there is no
virus and only uninfected lymphocytes. In addition, we call
seropositivity the state where viruses coexist with lymphocytes ($V
\neq 0$). Notice that the link of our denomination with what is
usually called ``seropositivity'' is not straightforward. We also
define a seropositivity fixed point to be admissible if the fields are
positive.

We gather here some mathematical study of well-known
models. These results are already known thoughout the
literature. Indeed one may find in the article of de Leenheer and
Smith \cite{Leenheer_Smith}, and an extension of Wang and Li
\cite{Wang_Li}, a very elegant way to study some of these models. In
these articles, the authors use an abstract characterization of the 3D
systems of ordinary differential equations that enable to have in a very elegant way the nature of
the fixed points, limit cycles and stability thanks to general results
on ordinary differential equations. They crucially use the decoupling in their 3D
analysis that {\em a priori} cannot apply to any fully 4D system.

A thorough study of numerous models is also done in
\cite{Callaway_Perelson}. In this article, the authors point out that
the fixed point viremia ($V^*$) is too dependent on the drug efficacy
and that intermediate levels of virus are too low to be realistic
($10^{-10}$ and so). They try numerous modifications, most of which do
not improve the behavior. But they provide some compartiment-like
models that do not support those two critics.

\bigskip     \noindent {\bf \large 2-1   \quad  Perelson's model}  

\smallskip  
In the review \cite{perelson02} (and numerous other papers with
various coauthors), A. Perelson proposes a dynamical system to
describe the interaction of HIV virus with CD4 lymphocytes. The model
uses four fields that we will denote with our notations to ease
comparisons: uninfected lymphocytes ($T$), infected lymphocytes ($U$),
infectious {\em free} viruses ($V$) and uninfectious {\em free}
viruses ($W$). After non-dimensionnalizing with respect to the amount
of lymphocytes in a safe body (and a given volume), the system reads:
\moneq  \label{syst_perelson}
\left\{
\begin{array}{rcl}
\Frac{{\rm d}T}{{\rm d}t} & = & \beta (1-T) -\delta_P V T \\[2mm]
\Frac{{\rm d}U}{{\rm d}t} & = & \delta_P V T - \alpha U \\[2mm]
\Frac{{\rm d}V}{{\rm d}t} & = & a \theta U - \sigma_P V \\[2mm]
\Frac{{\rm d}W}{{\rm d}t} & = & a (1-\theta) U - \sigma_P W
\end{array} \right. \monend 
In this system, when the parameters used by Perelson appear in terms
that do not appear in our model, they are denoted with his notation
with an index $P$. This model is already studied for
instance in an article of Nowak and Bangham of 1996
\cite{Nowak_Bangham}.

  \bigskip   \noindent {\bf  2-1-1   \quad  Fixed points}

One may state the following theorem concerning the fixed points of
(\ref{syst_perelson}).
\begin{theorem} \label{th_perelson_1}
There exists only two fixed points to system (\ref{syst_perelson}).
The first one is ``health'': $(T^*,U^*,V^*,W^*)=(1,0,0,0)$. The
second one (seropositivity) reads:
\moneq \label{eq1}
T^*= \frac{\alpha \sigma_P}{a \delta_P \theta}, \quad V^* = 
\frac{\beta}{\delta_P}\left(\frac{1}{T^*}-1\right), \quad
W^* = \frac{1-\theta}{ \theta} V^*, \quad U^* = \frac{\sigma_P}{a \theta} V^*. 
\monend 
The seropositivity is admissible ($V \geq 0$) under the condition that
\moneq \label{eq2}
a \delta_P \theta -\alpha \sigma_P >0.
\monend  
\end{theorem}
 The proof of this theorem is easy and left to the reader.

   \bigskip  \noindent {\bf  2-1-2   \quad   Stability of fixed points}

So as to evaluate the local stability of fixed points, one must compute the 
Jacobian matrix ${\rm d}F$ of the right hand side:
\moneq  \label{eq3}
{\rm d}F(T^*,U^*,V^*,W^*)= \left(\begin{array}{cccc}
-\beta-\delta_P V^* & 0 & -\delta_P T^* & 0 \\
\delta_P V^* & -\alpha & \delta_P T^* & 0 \\
0 & a \theta & -\sigma_P & 0 \\
0 & a(1-\theta) & 0 & -\sigma_P
\end{array}\right). \monend %
In the case of ``health'', $(T^*,U^*,V^*,W^*)=(1,0,0,0)$ it looks:
\moneq  \label{eq4}
{\rm d}F(1,0,0,0)= \left(\begin{array}{cccc}
-\beta & 0 & -\delta_P & 0 \\
0 & -\alpha & \delta_P & 0 \\
0 & a \theta & -\sigma_P & 0 \\
0 & a(1-\theta) & 0 & -\sigma_P
\end{array}\right). \monend %
 It enables to state a theorem of conditional stability for "health".
\begin{theorem}
\label{th_perelson_2}
Health as a fixed point is stable if and only if 
\moneq  \label{eq5}
a \delta_P \theta - \alpha \sigma_P <0.  
\monend  \end{theorem}

\noindent
The proof is very simple and left to the reader. The eigenvalues are
real and take the values $-\beta, -\sigma_P$, $1/2(-\alpha -\sigma_P
\pm \sqrt{(\alpha-\sigma_P)^2+4a \delta_P \theta}))$. The
admissibility of the unstable direction can also be checked.\\

\noindent
In the case seropositivity may occur, one may state the following theorem.
\begin{theorem}
\label{th_perelson_3}
Under the admissibility assumption
(\ref{eq2}), the seropositivity fixed point (\ref{eq1}) is stable.
\end{theorem}
%
 
\noindent $\bullet$ \quad  Proof of Theorem 3.  

\noindent Apart from $\lambda = -\sigma_P$, the eigenvalues satisfy:
\moneqstar  
\left| \begin{array}{ccc}
\lambda + \beta + \delta_P V^* & 0 & \delta_P T^* \\
- \delta_P V^* & \lambda + \alpha & - \delta_P T^* \\
0 & -a \theta & \lambda + \sigma_P
\end{array}
\right| = 0, 
\monendstar
or $\lambda^3+b_1 \lambda^2+b_2\lambda+b_3=0$ with
\moneq \label{eq6} 
b_1=\frac{1}{\alpha \sigma_P}(\beta a \delta_P\theta+\alpha \sigma_P(\alpha+\sigma_P)), 
\; b_2=\frac{\beta a \delta_P \theta}{\alpha \sigma_P}(\alpha +\sigma_P), \; 
b_3=\beta(a\delta_P \theta - \alpha \sigma_P).
\monend  
We need to use the Routh-Hurwitz Criterion (\cite{Henrici} p. 490)
that gives necessary and sufficient conditions ensuring that the roots
of the cubic polynomial have positive real parts. In our case, this
criterion reads:
\moneqstar   
\Delta_1=b_1 >0, \quad \Delta_2=\left|\begin{array}{cc} b_1 & 1\\ b_3 & b_2\end{array}\right| >0, \quad  
\Delta_3=\left|\begin{array}{ccc}
b_1 & 1   & 0 \\
b_3 & b_2 & b_1 \\
0   & 0   & b_3\end{array}
\right|>0. 
\monendstar
The first condition is obviously satisfied. Thanks to condition
(\ref{eq2}), the third condition is equivalent to the second one. It happens
that $\Delta_2$ can be computed:
\moneq \label{eq7}
\Delta_2 =\frac{1}{\alpha^2 \sigma_P^2} (\beta a \delta_P \theta 
+ \alpha \sigma_P(\alpha +\sigma_P))(a \beta \delta_P \theta (\alpha + \sigma_P)
+\alpha^2 \sigma_P^2)-\beta(a \delta_P \theta -\alpha \sigma_P),
\monend  
and this term is obviously positive because the only negative term is
compensated by one of the expanded terms.
\hfill $ \square $

  \bigskip     \noindent {\bf  2-1-3   \quad   Some numerical simulations}

 So as to have simulations comparable with
the other models studied in the present article, we
use the same values of parameters as previously: 
\moneqstar
 \beta = 0.01\mbox{
day}^{-1} \,,\quad \alpha = 0.7 \mbox{ day}^{-1}\,,\quad a = 250\mbox{
day}^{-1}.
\monendstar  
Moreover, some other parameters were the same as in the
Snedecor's model and we used the same values:
\moneqstar
\delta_P = 0.0125\mbox{ day}^{-1} \,,\quad \sigma_P = 2\mbox{ day}^{-1}.
\monendstar  
So as to be in the regime of health stable, we used
$\theta = 0.1$. As can be checked on
Figure~1, 
health is stable, although it needs
numerous days to come back to health. The qualitative evolution of
the lymphocytes and viruses
is comparable to the one of Snedecor's Model in
Figure~4. 
The uninfected lymphocytes decrease to
their minimum value (0.999625) in about six days. Then some hundreds
of days are needed to get sufficiently close to 1.

\bigskip  \centerline {
{\includegraphics[width=.49 \textwidth]{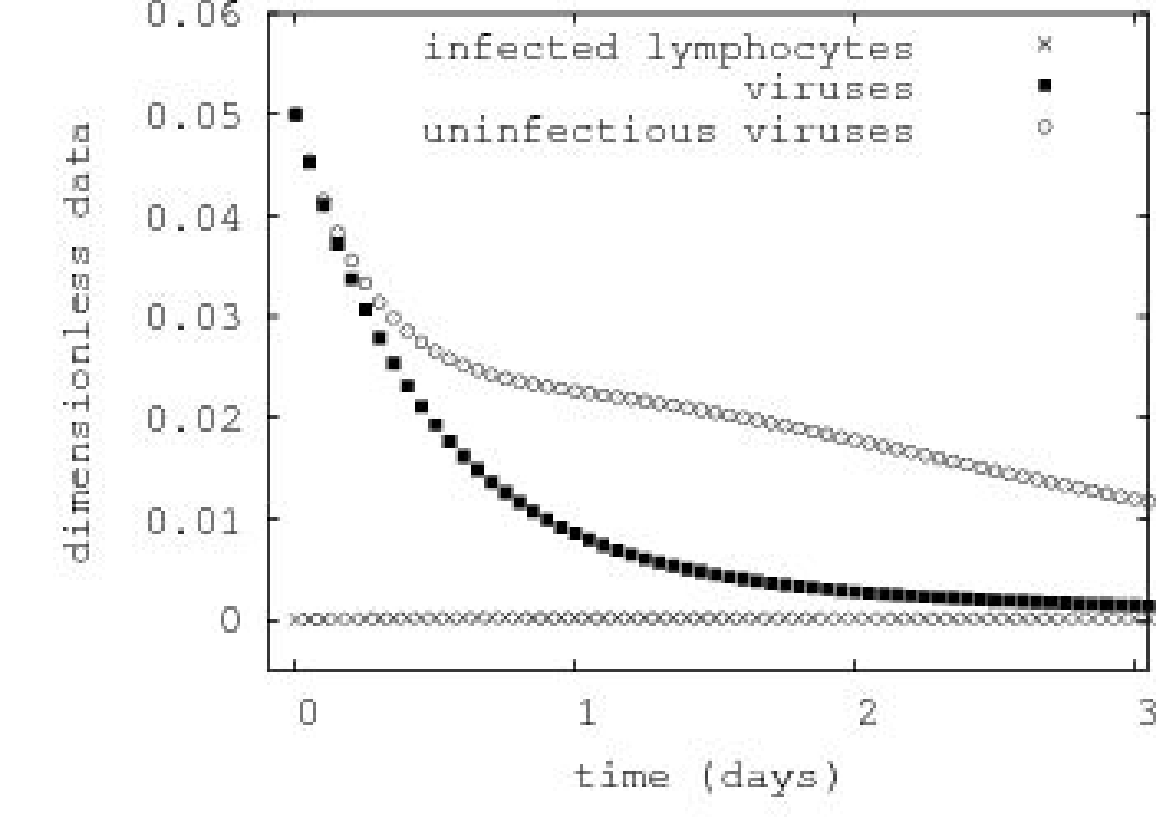}}  
{\includegraphics[width=.49 \textwidth]{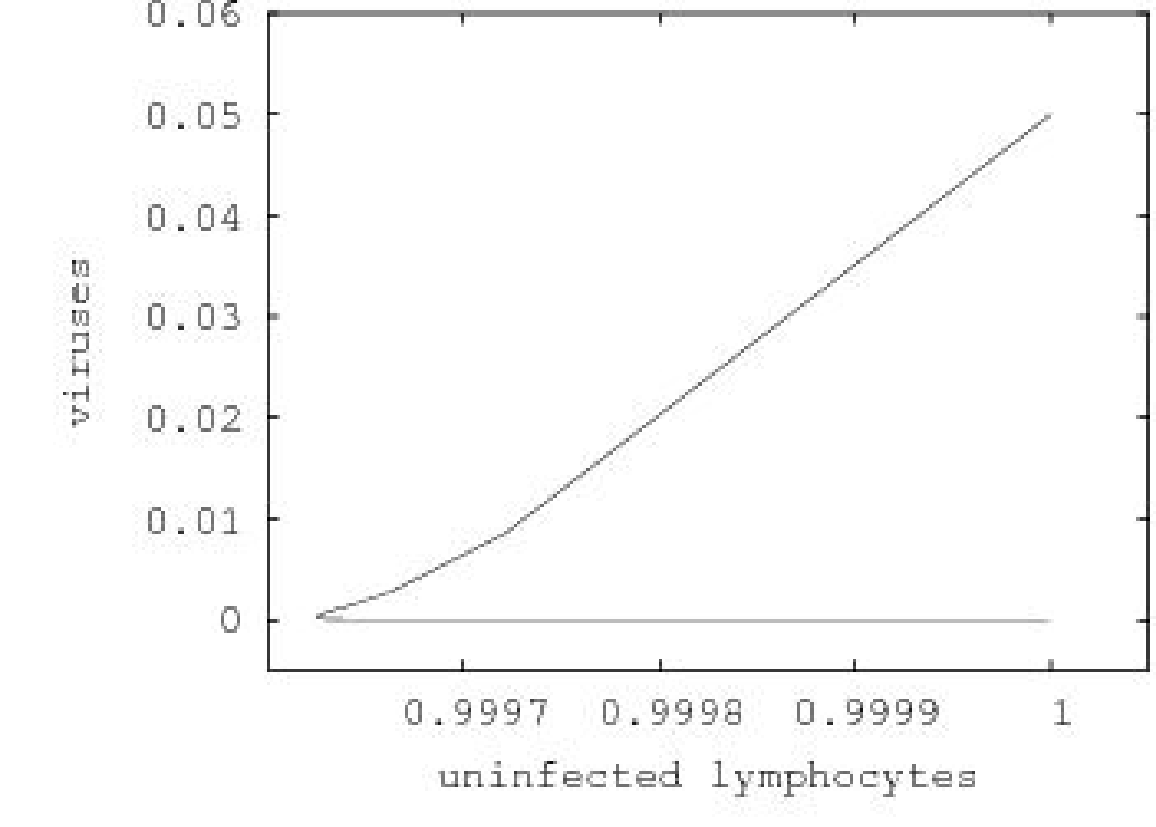}}} 
\hangindent=7mm \hangafter=1 \noindent {\bf Figure 1}. $\,$  
{\it Numerical results for Perelson's model with four equations.
Case where the health is stable. Short time evolution on the left (3 days) 
and phase plane for
lymphocytes and viruses on the right for a 600 days evolution.   }
\bigskip  

\noindent 
So as to have a stable seropositive state, we only change
$\theta$ for $\theta = 0.6$. Results can be seen on
Figure~2. 
The global dynamics converges to the
seropositivity with oscillations that could be a model of the
`blips'. The maximum value of the free virus (about 2.2) is reached at
day 39. Then the uninfected lymphocytes decrease until 0.599 at day
58. The minimum value of virus (0.019525) happens on the 116th day and
the cycle goes on as shown in Figure~2  right. 
Such oscillation are sometimes interpreted as the ``blips'' (sudden bursts
of viremia).

\noindent 
Notice that A.S. Perelson has discussed this model by proposing other
terms for the infection (equation (10) in \cite{Boer_Perelson}), the
immune system generation of T cells and of effectors, with and without
saturation \cite{Boer_Perelson}.

\bigskip  \centerline {
{\includegraphics[width=.49 \textwidth]{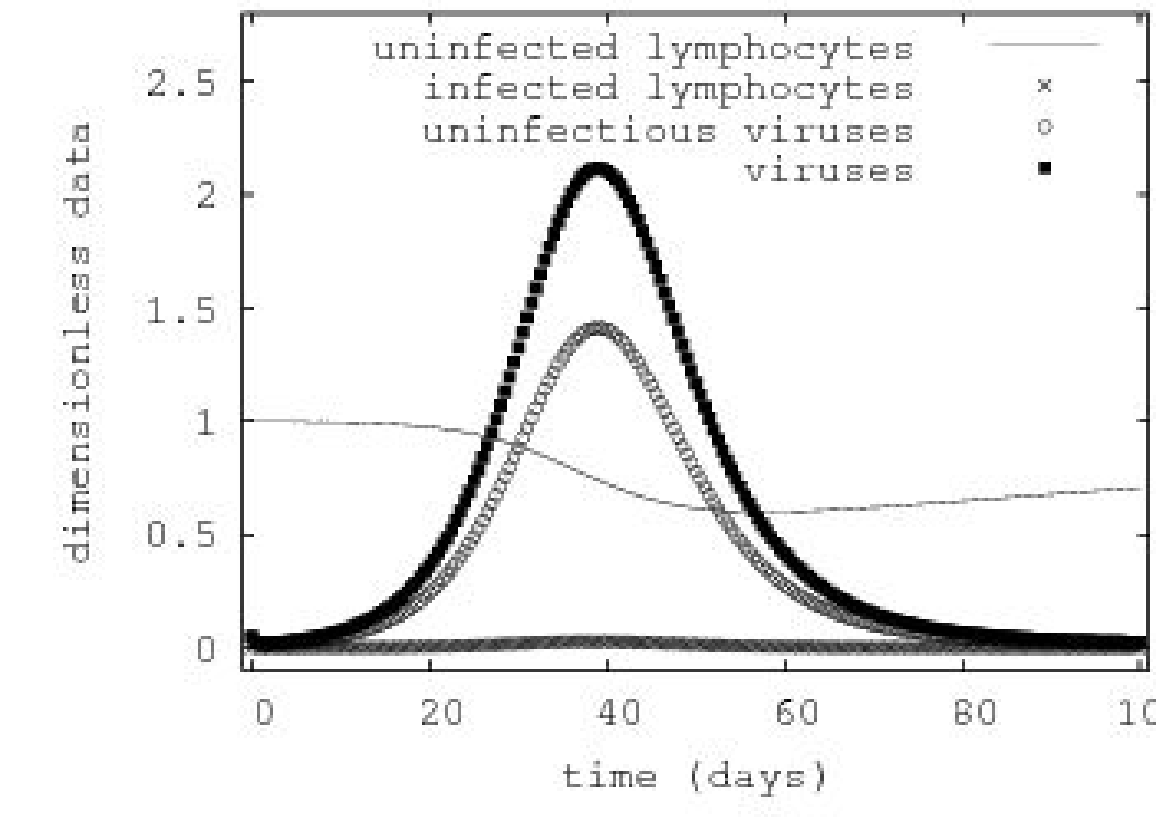}}  
{\includegraphics[width=.49 \textwidth]{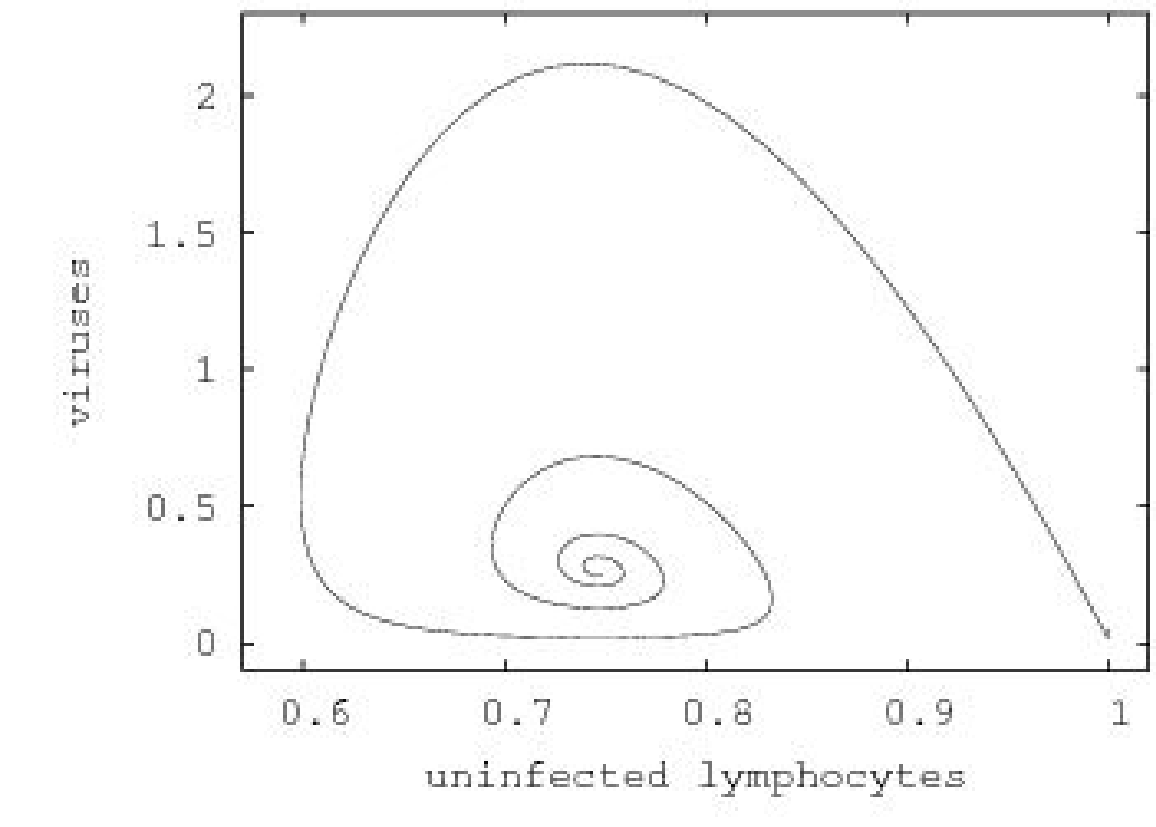}}}  
\hangindent=7mm \hangafter=1 \noindent {\bf Figure 2}. $\,$   
{\it Numerical results for Perelson model with four equations.
Case where the seropositive state is stable. 
Relatively short time evolution on the left  (200 days) and phase plane for
lymphocytes and viruses on the right for a 600 days evolution. } 
\bigskip  

\bigskip     \noindent {\bf \large 2-2   \quad  A reduced Snedecor's model}

\smallskip  
In \cite{Snedecor03}, the author gives three models of the
multiplication of virus HIV-1. Her goal is to model both drug
resistance, the different behaviors in the lymphatic tissue and the
peripheral blood, but also the immune system. Distinguishing the
blood and lymphatic tissues drives the author to have numerous
constants that model the fluxes between these body parts and enriches her
model.

In order to have a common basis with other models, we need to reduce the
above mentionned models to make their main features comparable with the
features of other models. After non-dimensionnalizing the three fields
with respect to the health value of lymphocytes ($T^*=2.5\, 10^{11}$),
we are lead to only one model in which the uninfected lymphocytes are denoted
by $T$, the infected ones by $U$ and the virus by $V$. We take the
values of the parameters in the same body part (lymphatic tissue). The
reduced Snedecor's model is written:
\moneq \label{syst_snedecor}
\left\{
\begin{array}{rcl}
\Frac{{\rm d}T}{{\rm d}t} & = & \beta (1-T) + 
\frac{r_S}{\gamma_S +V}(T-1)-(1-\alpha_S)\beta_S V T \\[2mm]
\Frac{{\rm d}U}{{\rm d}t} & = & +(1-\alpha_S)\beta_S V T -\alpha U \\[2mm]
\Frac{{\rm d}V}{{\rm d}t} & = & a U - \sigma_S V - \beta_S V T.
\end{array}
\right.
\monend 
In this system, when the parameters used by S. Snedecor appear in
terms that do not appear in other models, they are denoted with her
notation with an index $S$. For instance, the division rate of $T$
cells is $r_S=0.004 \mbox{ day}^{-1}$, the treatment efficacy
is $\alpha_S \in [0,1]$, and the viral clearance is $\sigma_S=2 \mbox{
day}^{-1}$. After non-dimensionalizing, the other parameters become
\moneqstar 
\beta_S=0.0125\mbox{ day}^{-1}, \gamma_S=4\times 10^{-5}, \beta =
0.01\mbox{ day}^{-1}, \alpha =0.7 \mbox{ day}^{-1}, \mbox{ and } a=250
\mbox{ day}^{-1}.
\monendstar 
%

  \bigskip     \noindent {\bf  2-2-1   \quad    Fixed points}

In the present subsection, we prove the following proposition:

\begin{proposition}
\label{fixed_points_S}
There exists a threshold
\begin{equation}
\label{def.s4}
\alpha_{S4}=1-\frac{\alpha(\beta_S+\sigma_S)}{a \beta_S}
\end{equation}
such that if the efficacy parameter $\alpha_S$ is above
$\alpha_{S4}$ then health is the only fixed point. If $\alpha_S <
\alpha_{S4}$ there are two fixed points: health and a seropositivity.
\end{proposition}

\noindent $\bullet$ \quad  Proof of Proposition 4.  

\noindent The search for fixed points gives two possibilities. 

\noindent
The first one is
health: $T^*=1, U^*=0,V^*=0$. 

\noindent
The existence of the second one depends on the therapy's efficacy
parameter $\alpha_S$. Different critical values of $\alpha_S$ will
appear in the discussion. The solution for $T$ is
\moneqstar 
T^*=\Frac{\sigma_S / \beta_S}{\frac{a(1-\alpha_S)}{\alpha}-1} \, ,
\monendstar 
and is drawn in Figure~3 
(left) with an infinite
value for $\alpha_S=\alpha_{S3}=1-\alpha /a=0.9972$. So, should a drug
be very efficient ($\alpha_S > \alpha_{S3}$), then $T^*<0$ and health
would be the only solution. Then the solution $V^*$ is a non-negative
solution of the second order equation:
\moneq      \label{equation_en_V_Snedecor}
V^2(1-\alpha_S)\beta_S T^*-V((r_S-\beta )(T^*-1)-(1-\alpha_S)\beta_S 
\gamma_S T^*)+\beta \gamma_S(T^*-1)=0.
\monend
Numerically, it seems that the discriminant is non-negative 
(see Figure~3  
right)  but indeed, it is negative between
$\alpha_{S1}\simeq 0.54920378$ and $\alpha_{S2}\simeq 0.5492747378$
and the minimum is about $-2 \times 10^{-13}$ ! Notice that when
perturbing the parameters this behavior remains. Except
for $\alpha_S$
not too close to 1, the discriminant is small (see Figure~3 
right). We have drawn in Figure~3 
(center) the only admissible solution $V^*$ as a function of $\alpha_S$.
It is then simple to have $U^*=(1-\alpha_S)\beta_S V^* T^* /\alpha$.

\bigskip  \centerline {
    \includegraphics[width=0.99 \textwidth]{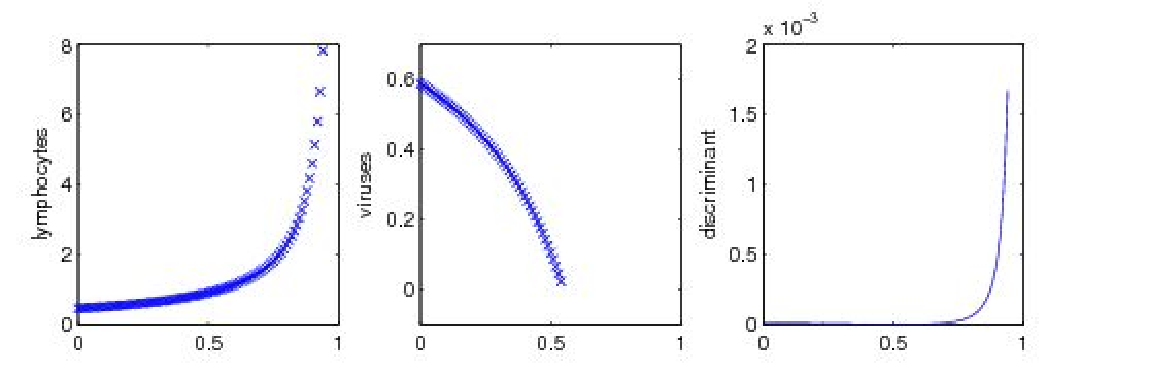}} 
 \hangindent=7mm \hangafter=1 \noindent {\bf Figure 3}. $\,$   
{\it The fixed points $T^*$ (lymphocytes), $V^*$ (virus) and the discriminant of
(\ref{equation_en_V_Snedecor}) as a function of the treatment efficacy $\alpha_S$.}  
\bigskip  \bigskip              

\noindent
If we take care of retaining only admissible solutions
($(T^*,U^*,V^*)\geq 0$), we must force $T^* \geq 0$ or $\alpha_S \leq
\alpha_{S3}$. Moreover, if $T^*$ crosses $1$ (for
$\alpha_S=\alpha_{S4}= 0.5492$
exactly see (\ref{def.s4})), the sign of the constant term in
(\ref{equation_en_V_Snedecor}) changes and so does one solution $V$ of
(\ref{equation_en_V_Snedecor}). One may then summarize the discussion
for $T^*$, and the two solutions $V^*_1$ and $V^*_2$ of
(\ref{equation_en_V_Snedecor}) in Table \ref{table1}. Indeed, there
exists a non-health solution only for $\alpha_S \leq \alpha_{S4}$. Such a
fixed point can be named seropositive solution.
\hfill $ \square $

\bigskip    \bigskip      
\centerline{ 
\begin{tabular}{r|c|c|c|c||c|}
$\alpha$ & \multicolumn{5}{c}{$\begin{array}{lccccr} 0 \hspace*{0.6cm} & 
\alpha_{S4} \hspace{6mm} &  \alpha_{S1}  & \hspace*{0.5cm} \alpha_{S2} & \hspace*{0.7cm} 
\alpha_{S3} & \hspace*{0.6cm} 1 \end{array}$}\\
\cline{2-6}
$T^*$ & \hspace*{0.3cm} + \hspace*{0.3cm}& \hspace*{0.3cm} 
+ \hspace*{0.2cm}& \hspace*{0.3cm} + \hspace*{0.3cm}& \hspace*{0.3cm} 
+ \hspace*{0.2cm}& \hspace*{0.2cm} $-$ \hspace*{0.2cm}\\
\cline{2-6}
$T^*-1$ & $-$ & + & + & + & $-$ \\
\cline{2-6}
$V^*_1$ &  $-$ & $-$ & $V^*_1 \in \mathbb{C}$ & $-$ & $+$ \\
\cline{2-6}
$V^*_2$ &  $+$ & $-$ & $V^*_2 \in \mathbb{C}$ & $-$ & $+$\\
\cline{2-6}
solution & $(T^*,V^*_2)$ & \O &\O  & \O & \O \\
\cline{2-6}
\end{tabular} }

$ \, \, $ 

\centerline { {\bf Table  1}. \quad   \label{table1} 
{\it Second fixed point existence for reduced Snedecor's model.} }
\bigskip      

  \bigskip \newpage     \noindent {\bf  2-2-2   \quad   Stability}

Since the domain of admissible fixed points is not regular,
the study of stability through the eigenvectors and eigenvalues at a
corner (such as health) needs to be more precise. We define
admissible directions to prohibit directions that do not enter the domain:
\begin{definition}
\label{admiss}
Let $ X \mapsto f(X)$ a smooth function and the associated dynamical
system $X'(t)=f(X)$. Let us denote $\mathcal{B}$ the biologically
admissible domain (all biological fields non-negative). Let $X^*$ be a
fixed point of a dynamical system ($f(X^*)=0$) at the boundary
$\partial \mathcal{B}$, and $(\lambda, u)$ one of its eigenvalue/
eigenvector.

The eigenvector $u$ is defined as admissible only if $\lambda>0$ and
either $X^* +\varepsilon u$ or  $X^* -\varepsilon u$ for positive
$\varepsilon$ enters the domain.
\end{definition}
The following proposition investigates the fixed points stability:

\begin{proposition}
\label{stability_S}
Let $\alpha_{S4}$ as defined in (\ref{def.s4}). When health
is the only fixed point ($\alpha_S > \alpha_{S4}$) it is
stable. When there are two fixed points ($\alpha_S < \alpha_{S4}$),
health is unstable in one admissible direction.
\end{proposition}

\noindent $\bullet$ \quad  Proof of Proposition 6. 

\noindent The Jacobian matrix of the second member of (\ref{syst_snedecor})
enables us to study the local behavior of the solutions. It is:
\moneq      \label{jacobienne_snedecor}
\left( \begin{array}{ccc}
\Frac{r_S V^*}{\gamma_S+V^*}-\beta-(1-\alpha_S )\beta_S V^* & 0 & 
-(1-\alpha_S)\beta_S T^*+\Frac{r_S \gamma_S (T^*-1)}{(\gamma_S+V^*)^2} \\
(1-\alpha_S )\beta_S V^* &-\alpha & (1-\alpha_S )\beta_S T^* \\
-\beta_S V^* & a & -\sigma_S - \beta_S T^*
\end{array} \right).
\monend  
In the case of the first fixed point (health $(T^*,U^*,V^*)=(1,0,0)$),
the characteristic polynomial is
$-(\beta+\lambda)(\lambda^2+\lambda(\alpha+\sigma_S+\beta_S)+\alpha(\sigma_S+\beta_S)
-(1-\alpha_S)a\beta_S)$. The discriminant of the second order
polynomial is $(\alpha-\sigma_S-\beta_S)^2+4(1-\alpha_S)a\beta_S
>0$. So the roots are real and it is possible to discuss the sign of the
roots. The already met threshold value $\alpha_{S4}=0.5492$
(exactly see (\ref{def.s4})) is stil the key value for discussing on $\alpha_S$.
If $\alpha_S > \alpha_{S4}$, health is alone
({\em cf}. Theorem \ref{fixed_points_S}) and the roots are negative, so
it is locally stable. If $\alpha_S < \alpha_{S4}$, there are two
stable and one unstable eigenvector. We easily check that the unstable
eigenvector at this fixed point (health), located at a corner of the
domain, enters the domain in the sense that one direction along this
eigenvector (among the two) lets all the fields be non-negative. So
this direction of instability is admissible.
\hfill $ \square $
 
\bigskip  \noindent 
We checked numerically that the locally stable fixed point ($\alpha_S
> \alpha_{S4}$: health alone) remains stable even under non-small
perturbations. In Figure~4,  
we draw the dynamics of stable health. 
In this computation, we take the same parameters as Snedecor:
\moneqstar    
\begin{array}{c}
r_S=4\times 10^{-3}\mbox{ day}^{-1}, \sigma_S=2\mbox{
day}^{-1}, \beta_S=1.25\times 10^{-2}\mbox{ day}^{-1}, \gamma_S=4
\times 10^{-5}, \\
\beta = 0.01\mbox{ day}^{-1}, \alpha = 0.7\mbox{
day}^{-1}, a = 250\mbox{ day}^{-1},
\end{array}
\monendstar 
\noindent 
and for the specific case of
health stable, we take $\alpha_S=0.6> \alpha_{S4}$ for the treatment
efficacy.
The qualitative evolution of the pair lymphocytes-virus is comparable 
to the one of Perelson's model presented on Figure~1. 
The number of virus loses a factor 5 in 3 days typically. 

\bigskip  \centerline {
{\includegraphics[width=.49 \textwidth]{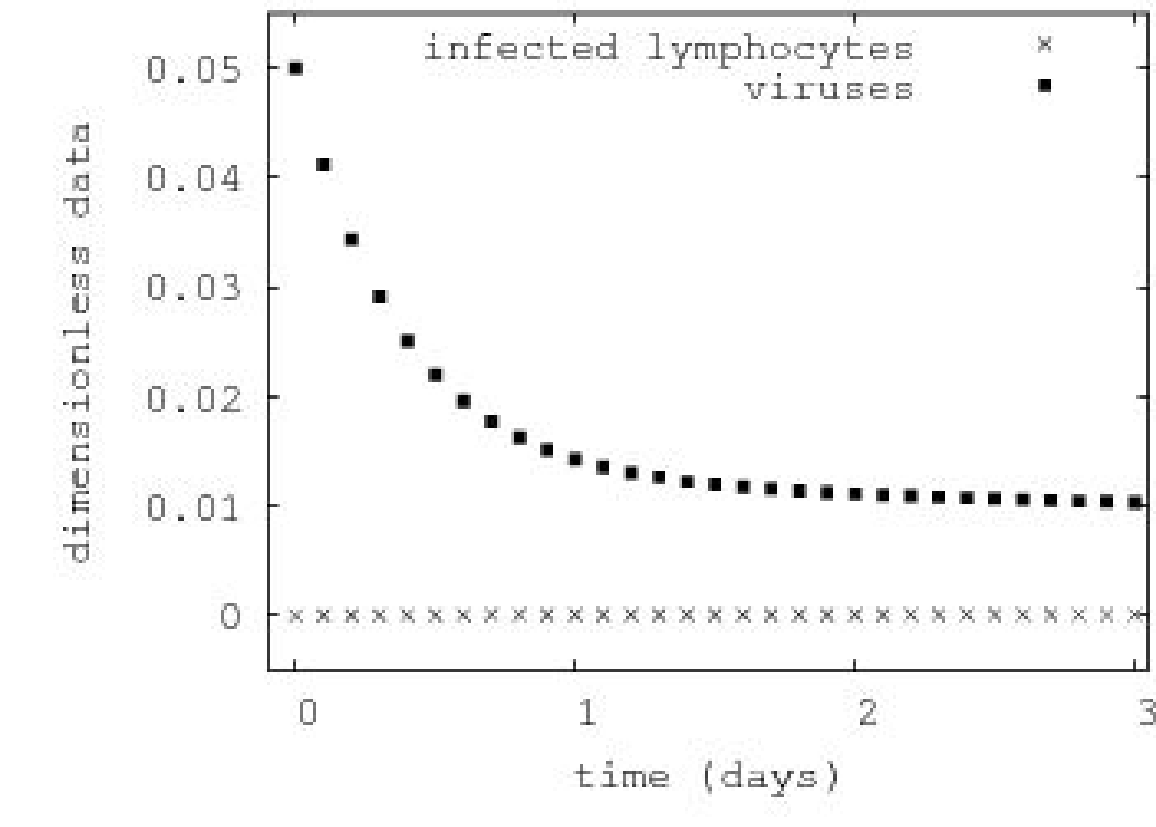}}  
{\includegraphics[width=.49 \textwidth]{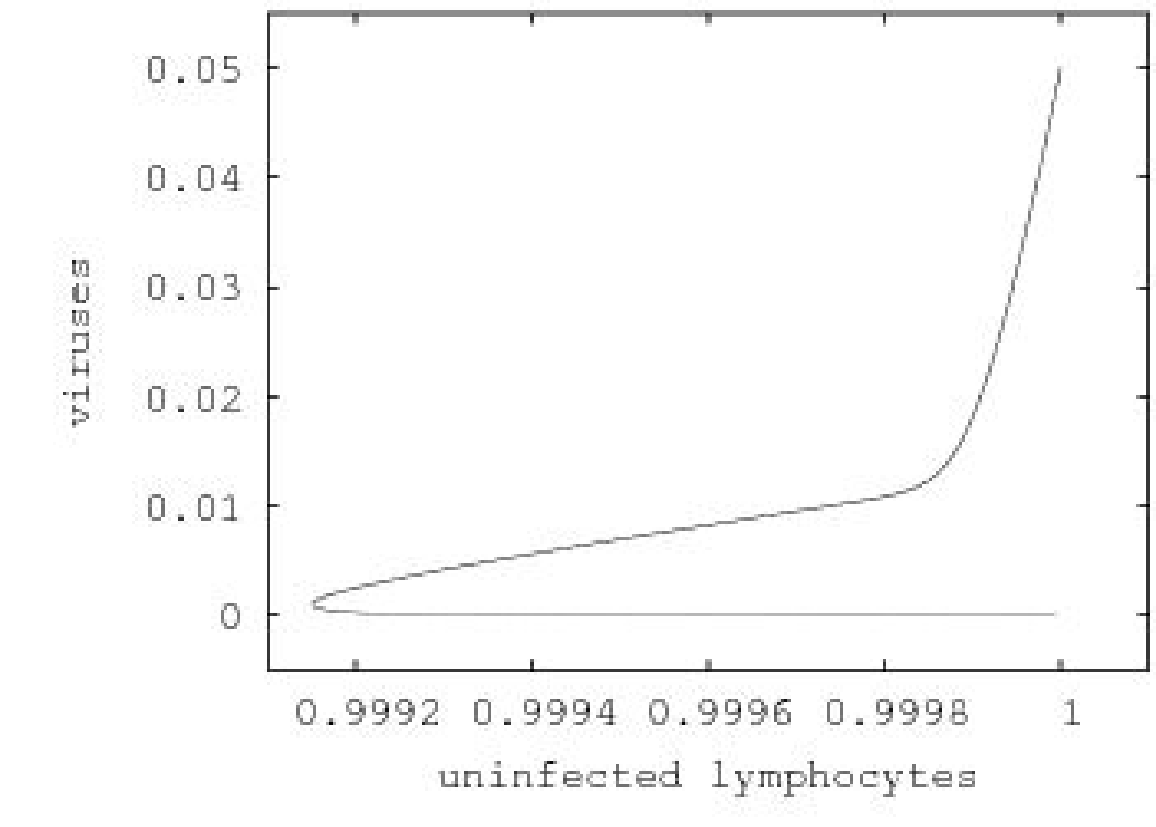}}}  
\hangindent=7mm \hangafter=1 \noindent {\bf Figure 4}. $\,$   
{\it  Numerical results for Snedecor model with three equations.
Case where the health is stable. Short time evolution on the left (3 days) and phase plane for
lymphocytes and viruses on the right for a 600 days evolution. } 
\bigskip  

\noindent 
In the case where there exists also a second fixed point
(seropositivity: $\alpha_S < \alpha_{S4}$), even odd initial
conditions like $(T_0,U_0,V_0)=(1,1,1)$ lead to the second fixed point
available. This second fixed point happens to be numerically locally
stable as can be seen on Figure~5 
where the parameters are the same as Snedecor's, as recalled above, except the
therapy efficacy $\alpha_S=0.3$. In this figure, one may check that
health is locally unstable as the solution evades health to get closer
to a seropositivity. Notice that this second fixed point can be
interpreted as seropositivity. But as it is stable, the model predicts
no death ...  The oscillations around the fixed point could be seen as
the observed `blips' (sudden and brief bursts of viremia). But the
time scale at which the system is sufficiently close to the fixed
point is more than one year while in reality, the first phase of the
infection takes some weeks.

\bigskip  \centerline {
{\includegraphics[width=.49 \textwidth]{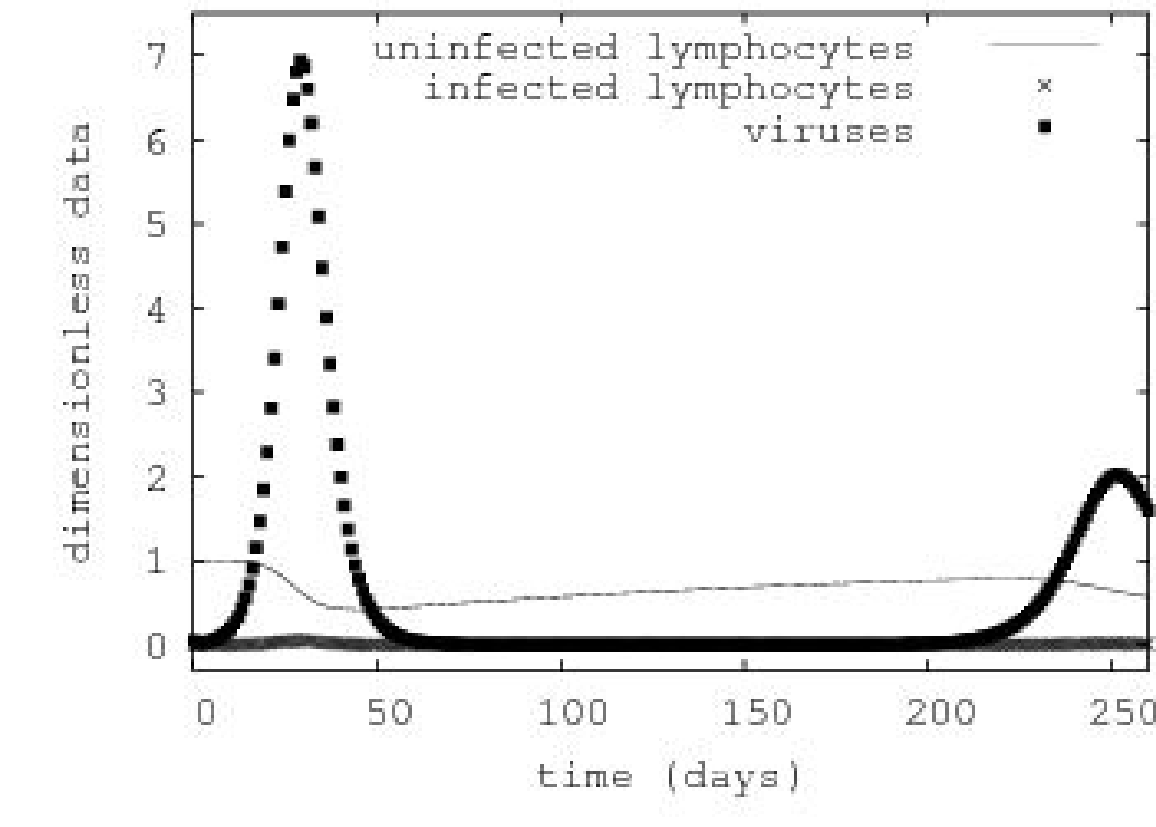}}  
{\includegraphics[width=.49 \textwidth]{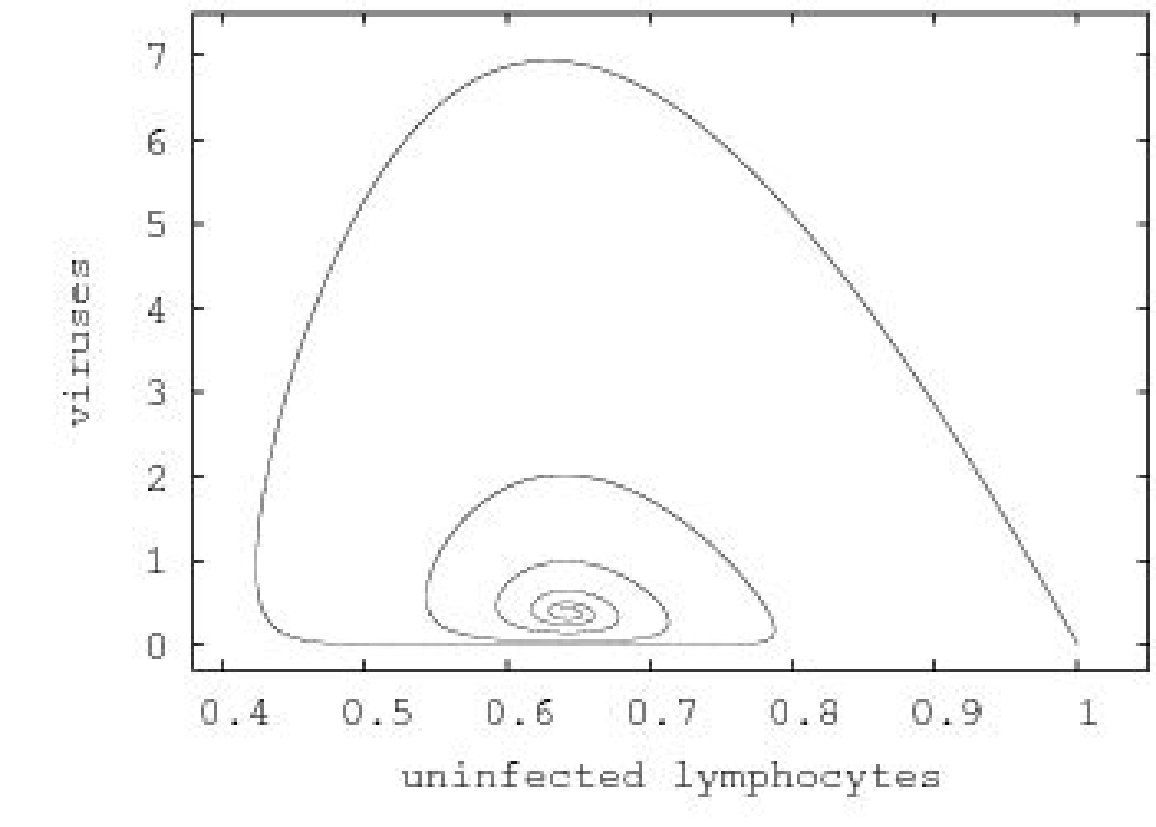}}}  
\hangindent=7mm \hangafter=1 \noindent {\bf Figure 5}. $\,$   
{\it  Numerical results for Snedecor model with three equations.
Case where the seropositive state is stable. Relatively short time evolution on the left 
(260 days) and phase plane for
lymphocytes and viruses on the right for a 900 days evolution. }
\bigskip  \bigskip              

  \bigskip       \noindent {\bf  2-2-3   \quad  Some comments}

The disappearance of the second fixed point (seropositivity) when
$\alpha_S$ increases could be meaningful. Yet in real life, even for
very efficient drugs, the virus kills. Despite highly active
multitherapies, over ten percent of AIDS patients face therapeutic
escape due to viruses which were indetectable for years while
developing resistance against all approved drugs.
So, to what extent can {\em any} stable fixed point be
meaningful ? This criticism is valid for all the models analyzed in this
article.

The original model makes a distinction between the drug-resistant and
drug-sensitive viruses which is
meaningful. But it makes no difference between infectious and
non-infectious viruses.

In a sense, the Snedecor's model takes into account the immune
system. But it was proved that ``the growth fraction of CD4$^+$ (...)
was correlated (...) with viral load'' \cite{Sachsenberg}. No such
correlation appears in her model. Indeed, the immune system is modeled
only through its exhaustion when $V$ is too large. The
susceptibility to produce $T$ in the presence of virus is not modeled
as it is done in \cite{Frid_Jabin_Perthame} with a term $\Sigma(V/T)T$.

If $\alpha_S=0$ (no treatment), the coefficient before the terms of
disappearance of $T$, of $V$ and appearance of $U$ is the same. The
identity of the $V$ and $T$ terms indicates that in this model, each
time a virus infects a $T$, it disappears. So viruses are assumed to be
{\em free viruses}. Indeed, the author defines her field $V$ as ``free
virus''. Similarly, the term $-\beta_S V T$ in the evolution
of $V$ is by no means a model of the immune system response. Moreover, the
author defines parameter $\beta_S$ as ``infection rate of ... $T$
cells by ... virus''. We emphasize here the coherence of Snedecor's
model which takes into account ``free virus'': a free virus {\em
disappears} each time it infects a lymphocyte. This property is not
satisfied by most other models as it is stressed in \cite{Leenheer_Smith} which emits the
same critic (p. 1314) but argues that this neglection does not change the main features of
the models.

In this model, viruses disappear only through natural death or infection
of a $T$. So there is no immune system effect. Indeed, $\sigma_S$ is
called ``viral clearance (death) rate'', but it does not depend on
$T$. This is frequent in HIV modelling, but not realistic unless the
immune system is neglected ! This is discussed in \cite{perelson02}
(pp. 31-32) where the author acknowledges his trouble: ``The fact
that models with constant [$\sigma_S$] can account for the kinetics of
acute HIV infection is surprising''.

\bigskip     \noindent {\bf \large 2-3   \quad  A Nowak and May model}

\smallskip  
In \cite{Nowak-May}, the authors also propose some models
taking the antigenicity into account. In Chapter 12 of their book,
they propose various models, suggesting that ``antigenic variation
generates the long-term dynamics that give rise to the overall pattern
of disease progression in HIV infection'' (p. 124). Their ``general
idea was that the rapid genetic variation of the virus generates over
time viral populations (quasispecies) which are more and more adapted
to grow well in the microenvironment of a given patient'' (p. 124).

According to their best model, ``The immune system and the virus
population are in a defined steady-state only if the antigenic
diversity of the virus population is below a certain threshold
value. If the antigenic diversity exceeds this threshold [then]
the virus population can no longer be controlled by the immune
system.'' (\cite{Nowak-May} p. 125). Such a behavior seems
meaningful. With our notations their model writes:
\begin{eqnarray*}
\dt{V_j}{t} & = & V_j(r_{NM}-p_{NM}T_j-q_{NM} Z), \; \forall i=1, ...N\\
\dt{T_j}{t} & = & c_{NM} V_j -b_{NM} T_j -U\, V \, T_j,\; \forall i=1, ...N\\ 
\dt{Z}{t} & = &  k_{NM} V -b_{NM} Z -u_{NM} \, V \, Z.
\end{eqnarray*}
where $V=\sum_j V_j$ and $Z$ denotes the ``cross-reactive immune
response directed against all different virus strains''
(\cite{Nowak-May} p. 130).

In the models they study, they set various parameters to values
assumed to be constant. More precisely, they assume the parameters do
not depend on $N$. Why is it impossible ? Assume there were 
only linear and non-linear terms like:
\moneqstar 
\dt{V_j}{t} =  r_{NM} \, V_j-p_{NM} \, T_j \, V_j \; \forall j.
\monendstar 
What is the integrated (in $j$) equation ? If we assume all the
populations $T_j$ and $V_j$ are independent on $j$, then $T_j=\sum_k T_k /N = T/N$ and
$V_j=V/N$. Then the integrated (in $j$) equation is:
\moneqstar  
\dt{V}{t} =  r_{NM} \, V -\frac{p_{NM}}{N} \, T \, V \hspace*{1.5cm} \mbox{ if }
T_j=\frac{T}{N}, V_j=\frac{V}{N}.
\end {equation*} 
Since it is the only equation biologically measurable, the measured
parameter in an experiment will be $p_{NM}/N$ and not $p_{NM}$. In
other words $p_{NM}=O(N)$. This should modify the mathematical study.

The conclusion of their study is that {\em in their
models} ``the cross-reactive immune responses provide a selection
pressure {\em against} antigenic variation, while strain-specific
responses select {\em for} antigenic variation'' \cite{Nowak-May}.

\bigskip \bigskip   \noindent {\bf \Large 3) \quad   A new model with antigenic variable }
  
\smallskip  
The biological diversity of antigenicities is already modeled in
\cite{Althaus_Boer,Frid_Jabin_Perthame,Nowak-May,Pastore_Zubelli}. Yet,
in these references, the authors use a finite number of possibilities,
most of the time in ``flat'' spaces of antigenicities. Since the
antigenicity is a ``microscopic'' and unmeasurable quantity, these
models (including ours) still need to prove they provide a significant
insight.

So as to build up our model, we will follow the biological description
of the various phenomena concerning the various fields: we denote $T_i(t)$ the
uninfected lymphocytes of antigenicity $i \in \mathcal{A}$, $U_i(t)$
the infected lymphocytes of antigenicity $i$, $V_i(t)$ the
infectious viruses of antigenicity $i$ and $W_i(t)$ the non-infectious
viruses of antigenicity $i$. The space $\mathcal{A}$ is still
undetermined. The best set is not known, but when trying to get a
macroscopic model (by integration of antigenicity $i$), we will need
to investigate the various possibilities to get a limiting operator
for $(t,i) \in \mathbb{R}^+ \times \mathcal{A}$. This is the issue of
a forthcoming paper. We will also need the sum of each field:
\moneqstar  
T=\sum _j T_j, U= \sum_j U_j, V=\sum_j V_j \mbox{ and } W= \sum_j W_j.
\monendstar

\bigskip     \noindent {\bf \large 3-1   \quad  Lymphocytes evolution}

\smallskip    
The variation in time of $T_j(t)$ (${\rm d}T_j / \dd t$) must
take into account various phenomena:

\monitem a natural death and generation modeled by a term like:
\begin{equation}
\label{modele_eq4.5}
\left(\frac{{\rm d}T_j}{{\rm d}t} \right)_{\mathrm{natural}}= -\beta_j T_j +\gamma_j
\end{equation}
It has also been proposed a logistic term (for instance in
\cite{Boer_Perelson}), but it prohibits high levels of T cells.

\monitem when a virus ($V_j$ or $W_j$) of antigenicity $j$ is
detected, the immune system generates lymphocytes of the same
antigenicity to fight them.  This can be modeled by an exponential
multiplication whose time constant is roughly proportional to the
inverse of the number of viruses $V_j+W_j$. This is a
``Lotka-Volterra'' type term that we met also in
\cite{Nowak_Bangham,Pastore_Zubelli,Boer_Perelson} for the
cross-antigenic immune action (CTLs or effectors):
\moneq \label{modele_eq2}
\left(\frac{{\rm d}T_j}{{\rm d}t} \right)_{\mathrm{growth}}= +C_j(V_j+W_j) \, T_j
\monend
The generation of lymphocytes by the immune system is studied in
\cite{Wodarz_Hamer}. The authors model it as a logistic growth $+C V T
(1-T/T_{max})$ (also discussed critically in \cite{Callaway_Perelson})
to represent the ``autocatalytic cell division''. But surprisingly the
effect of the immune system against virus or infected lymphocytes is
neglected. They conclude that the ``initial number of the HIV-specific
CD4-T cells'' is crucial. Yet, it depends on the definition of the
initial time. At the very first time, it is necessarily small (for 5
liters of blood). Such a logistic term (whatever the infection term and whever it models
the natural growth or the immune system reactivity) would force the
immune system not to have any overshoot of T cells ($T>
T_{max}$). This is not realistic.

In their article \cite{Boer_Perelson}, de Boer and Perelson investigate the
immune system production of T and infection modeling. In this central
article, but on macroscopic models, they propose a term like ours
and even saturate it (their (13)). As they conclude ``one may model
disease progression by allowing the virus to evolve immune-escape
variants increasing the diversity of the quasi-species [references
proposed]. Since this requires high-dimensional models, this form of
disease progression is not considered any further here''
(p. 208). Such a study is the goal of the present article.

In \cite{Frid_Jabin_Perthame}, the authors model ``antigenicity'' through
a real variable. At first glance, one may believe that their term
$\Sigma(\, (\eta +V)/T)T$ (with our
notations), where $\eta$ is the resource, is like our's for infection. Although the
$\Sigma$ function has the same properties as our $J$ (see later), our term models
infection, while their's models the reactivity of the immune
system. So they even have opposite sign. Their term should be
compared to our $C_j(V_j+W_j)T_j$
except that they bound the importance of $V$ for production of $T$, so
modeling a kind of exhaustion of the immune system very interesting.

the viruses attack lymphocytes independently of the
antigenicity. More precisely, there are two regimes:


\sonitem If $V/T \le 1$ as a virus may attack only one lymphocyte $T_k$ \cite{Dern}, the
ratio of attacked lymphocytes is $V/T$. 
If we introduce  a time constant, it gives a term like:
\moneqstar   
\left(\frac{{\rm d}T_j}{{\rm d}t} \right)_{\mathrm{infection}} \simeq -\Frac{1}{\tau_j
}\left(\Frac{V}{T} \right) T_j, \quad \mbox{ if } \frac{V}{T} \leq 1.
\monendstar   
\sonitem If $V/T \ge 1$, as a virus may infect only one lymphocyte \cite{Dern}
only part of the viruses may infect the lymphocytes. More precisely, no more
than $T$ viruses may infect. In other words, the ratio of attacked
lymphocytes may not be superior to 1. There is a kind of saturation of
the efficacy of predators $V_j$ while the number of lymphocytes
decreases with time. Up to a time constant, 
we saturate the infection ratio:
\moneqstar    
\left(\frac{{\rm d}T_j}{{\rm d}t} \right)_{\mathrm{infection}} \simeq
-\Frac{1}{\tilde{\tau_j} } T_j, \quad \mbox{ if } \frac{V}{T} \geq 1.
\monendstar   
\sonitem 
If $V/T\sim 1$, the two terms must match by continuity, and so $\tau_j=\tilde{\tau_j}$.
%
This complex behavior may be compiled due to a non-linear function $J$
of the ``min-mod'' type:
\moneq   \label{modele_eq3}
J(\xi) = \left\{ \begin{array}{cr} \xi & \hspace*{1cm} \mbox{ if } \xi \ll 1 \\
1 & \hspace*{1cm} \mbox{ if } \xi \gg 1 \end{array} \right. .
\monend  
The effect discussed here can be modeled by
\moneq   \label{modele_eq4}
\left(\frac{{\rm d}T_j}{{\rm d}t} \right)_{\mathrm{infection}} =-\Frac{1}{\tau_j
}J\left(\Frac{V}{T} \right) T_j.
\monend  
In order to justify our essentially linear
term in (\ref{modele_eq4}) in an other way, we make hereafter various assumptions
and wonder how our term should behave upon these assumptions, seen how reality behaves.

Firstly we assume that we double $V_j$ (and only it) without changing
$V$ or to a noticeable extent, while $V/T$ remains small. Then the effect does not
change and so the term must not change. This prohibits a simple term
like $-V_j /\tau_j$.

Secondly we assume that all the $V_i$ are doubled (including $V_j$)
and so $V$ is doubled too for a
still small $V/T$. Obviously the effect is doubled. So the
modeling term should depend on $V_i$ (whatever $i$) only through
$V$.

Thirdly we assume that we double $T_j$, and only it (not the more general
$T_i$), without changing
$T$ (or to a noticeable extent) nor $V$ (and still $V/T \ll 1$). Then the
effect doubles. It proves that the term should depend linearly on $T_j$.

At this stage, we have a linear dependance both in $V$ and $T_j$. If we
do not take care, we might deduce a $-V T_j/\tau_j$ term that leads to $-VT/
\tau$ when integrated over $j$. We still need to explain
why our term {\em may not} depend linearly on $T$ although it depends
on $T_j$.

To that end, we make the assumption that we
double {\em all} the $T_i$ (including $T_j$).
Since a virus may infect only one lymphocyte at a time, the effect
should not be modified as the
limiting parameter is not the amount of lymphocytes but the amount of
viruses. Indeed, assume a patient has caught any disease that makes his immune system produce
lymphocytes, it should not make the HIV infection more virulent, in the
first stage where $V/T \ll 1$. This is satisfied by our term and not by $-V T_j/ \tau_j$.

\smallskip  
Notice that by summing over $j$ and assuming $\tau_j$ constant, we
find $\dd T/\dd t= -V/ \tau$ in the first regime ($V\le T$) and $\dd
T/\dd t= -T/ \tau$ in the second regime ($V\ge T$). Moreover, if $V
\gg T$ the effect does not depend on $V$ as the limiting factor is the
presence of lymphocytes. These are expected as they seem natural.
Most authors have considered that the infection term should be
quadratic like $-V\, T/\tau$ (mass-action). This conclusion may not be
drawn from our considerations. It is even denied as the regimes
exclude one another.  It must be recalled that all the models are
designed to be tested with the only biologically meaningful fields
$T+U$ and $V+W$. Yet the phenomenon we model is modeled
by two terms of opposite signs that simplify in the evolution equation
of $T+U$. So it should never be measured in macroscopic fields.  Yet, it
translates into the model a crucial biological reality. So this term
would deserve to be tested and such a test is postponed to a
forthcoming research.

A term very similar to ours can be found in the article of de
Boer \cite{Boer} who uses a Michaelis-Menten kinetic: $-\beta T V
/(h+T+V)$ for the infection (and parameters $h$ and $\beta$). Such a
term saturates the infection when there are numerous $T+V$. But for
moderate or low $T+V$, the term is
essentially quadratic (mass-action). The authors critic models
``for acute HIV infection [that] tend to ignore saturations
effects and have simple ``mass-action'' terms ''. They also
give a toy model
\moneqstar   
\dt{V}{t} = (r-kT)V,
\monendstar   
for which the infection would get cleared as soon as $T> r/k$
``which is {\em independent} of the size of the pathogen
population ... which is entirely unrealistic ''. Their model
is extended in \cite{Althaus_Boer} to take into account ``specific parts of the viral
proteins, i.e., epitopes'' that
are the same as our antigenicities. But no real mutation is modeled.
Their equations do not model the immune system
effect against the virus.

An explicit discussion of such a term is also done in
\cite{Callaway_Perelson} (p. 37) where the authors discuss the
classical mass-action term for infection. As they say, this notion
``is valid when the system is well mixed [...] and there are
significant quantities of each reactant.''. They propose various terms
which have the same behavior as ours. When they try a Michaelis-Menten
like term (such as $kVT/(\alpha+V+T)$ or our term), it does not
prohibit their two main critics for admissibility of models: it reachs
too low viremia that should mean extinction and still has a strong
dependence of the fixed point $V^*$ on the drug efficacy. So they
reject such an infection term. Why this critic does not apply to our
model ? First their definition of $V$ is for {\em free} virus. So
vanishing viremia does not prove there is no more virus in the
body. Second the immune system is indeed sufficiently strong to
eradicate any {\em given} antigenicity. Only the endless mutation and
sanctuaries kill the patient. So ``macroscopic'' models, such as the
ones they consider, should not be able to model the race between the
immune system and the virus whose {\em differences} of velocities may
explain the slow eradication of lymphocytes. This is more likely to be
contained in ``microscopic'' models taking antigenicity and mutation
into account. Notice that the authors exhibit some models that do not
have the two main drawbacks. They are compartment-like models where
the drug is not efficient in a compartment. So, in a sense they are
rather similar to full our model with $N > 1$.

In \cite{Frid_Jabin_Perthame}, the authors model ``antigenicity'' through
a real variable. Their infection term is of the shape $VT/(1+V)$ 
which does not depend on antigenicity, models a saturation
effect on $V$ but not on $T$. 
 
\smallskip  
The consolidated evolution equation for the lymphocytes is the sum
\begin{eqnarray} \nonumber
\Frac{\dd T_j}{\dd t} & = & \left(\frac{{\rm d}T_j}{{\rm d}t}
\right)_{\mathrm{natural}}+\left(\frac{{\rm d}T_j}{{\rm d}t}
\right)_{\mathrm{growth}}+\left(\frac{{\rm d}T_j}{{\rm d}t} \right)_{\mathrm{infection}} \\
\label{modele_eq5}
\Frac{\dd T_j}{\dd t} & = &-\beta_j \, T_j+\gamma_j+C_j(V_j+W_j) \, T_j -\Frac{1}{\tau_j
}J\left(\Frac{V}{T} \right) T_j.
\end{eqnarray}

\bigskip     \noindent {\bf \large 3-2   \quad  Infected lymphocytes evolution }

\smallskip   
The evolution of infected lymphocytes depends on various effects: 

\monitem the infection of a lymphocyte by a virus (same term as for
$T_j$ (\ref{modele_eq4}) with a plus sign) generates an infected
lymphocyte $U_j$:
\moneqstar    
\left(\frac{{\rm d}U_j}{{\rm d}t} \right)_{\mathrm{generation}} =+\Frac{1}{\tau_j
}J\left(\Frac{V}{T} \right) T_j ;
\monendstar   
\monitem and a natural death:
\moneqstar    
\left(\frac{{\rm d}U_j}{{\rm d}t} \right)_{\mathrm{natural}} =-\alpha_j \, U_j.
\monendstar

Notice, that the death rate of infected lymphocytes is about 70 times
greater than that of uninfected lymphocytes. Indeed, in \cite{Snedecor03},
the author proposes various articles among which \cite{Finzi} for the
death rate of infected lymphocytes and \cite{Sachsenberg} for the
death rate of uninfected lymphocytes. The ratio of these rates is
about 70. 

To summarize, we have:
\moneq  \label{modele_eq6}
\Frac{\dd U_j}{\dd t}= \left(\frac{{\rm d}U_j}{{\rm d}t}
\right)_{\mathrm{generation}}+\left(\frac{{\rm d}U_j}{{\rm d}t}
\right)_{\mathrm{natural}}= +\Frac{1}{\tau_j }J\left(\Frac{V}{T} \right) T_j -\alpha_j \,
U_j.
\monend

\bigskip     \noindent {\bf \large 3-3   \quad   Infectious and non-infectious virus evolution }

\smallskip     
The multiplication of viruses is the main phenomenon and occurs in the
infected lymphocytes. So the sum in $j$ growth term of $V+W$ must
be 
\moneq \label{modele_eq6.5}
\sum_j \left(\frac{{\rm d}(V_j+W_j)}{{\rm d}t} \right)_{\mathrm{growth}}= \left( \frac{\dd
  (V+W)}{\dd t}\right)_{\mathrm{growth}}= a\,U,
\monend   
as taken into account by \cite{Nowak-May},
\cite{Snedecor03} and \cite{perelson02} (whether they distinguish infectious
and non-infectious virus or not).

We do not distinguish infecting and free viruses as the other authors do.
This would have led us to three identical terms (up to the sign). Two
terms for the infection of T lymphocytes in the evolution equations of $T$
(sign -) and of $U$ (sign +). One more term for the disappearance of a
virus in the evolution equation of $V$ (sign -). The latter is almost always
omitted. One potential justification for
this omission is offered by the authors of \cite{perelson99mathematical}
(p. 10) who argue that the term $k_i T_i V$ is
small in comparison with $c V$. In \cite{Nowak_Bangham}, the authors
say (note 20) that ``if a large number of virus particles is produced,
only a few of which will end up in host cells, then a constant death
term is a reasonable approximation''.

The evolution of virus depends also on mutations. Some of them modify
the ability to infect and others the antigenicity. Here, those two
properties are considered as independent, as mentioned in the
introduction where references are given.

Only one parameter $\theta$ measures the probability to mutate to an
infectious offspring (necessarily from an infectious virus). So
$1-\theta$ is the probability to mutate to a non-infectious
offspring. Since we assume antigenicity and ability to infect are
independent, we may assume $\theta$ does not depend on $j$ (nor
$k$).

Let us denote $S_{kj}$ the probability to mutate from an antigenicity
$k$ (only $V_k$ as $W_k$ does not even infect and so does not mutate)
to an antigenicity $j$ (either $V_j$ or $W_j$) per unit of time. So we
have:
\moneq         \label{modele_eq7}
S_{kj}\geq 0, \hspace*{2cm} \sum_j S_{kj}=1,
\monend   
and $S_{kj}$ depends {\em a priori} on the number of antigenicities. The offspring
of a virus $V_k$ will mutate to $V_j$ with the probability $\theta
S_{kj}$ (the case $k=j$ where there is no mutation is included). So it
will mutate to $W_j$ with the probability $(1-\theta) S_{kj}$.\\

\smallskip  
As a consequence of the assumption that antigenicity and ability to
infect are independent, equation (\ref{modele_eq6.5}) can be split into
two parts :
\moneq  \label{modele_eq8.25}
\left( \frac{\dd V}{\dd t} \right)_{\mathrm{growth}}= a \theta U, \quad \left( \frac{\dd
  W}{\dd t} \right)_{\mathrm{growth}}= a (1-\theta) U.
\monend   
 Notice that the value of $\theta$ suggested by \cite{Thomas}
ranges from 1 to 1/60,000. But \cite{Kothe}
found in one experiment a ratio $\theta$ close to 1/8. To fix
the ideas, we suggest in the following to set $\theta=1/10$.

The value of $S_{kk}$ is considered as relatively low by the biological
community but we have not found any precise and trust-worthy proposition in the
literature.

The distribution $k \mapsto S_{kj}$ for $j\neq k$ needs to
be non-uniform for a biological reason. As is well known (\cite{Bebenek}, \cite{Ji}),
there exist mutationnal hot spots in the 
HIV virus. We may guess these hot spots will elicit non-uniform
distribution of the antigenicity during multiplication/mutation. The
precise value of the probability $k\mapsto S_{kj}$ is of course an
open problem.

   \bigskip  \newpage \noindent {\bf   3-3-1  \quad    Infectious virus}

\smallskip  
The infectious virus variation ($\dd V_j / \dd t$) depends on various
effects. 
  
\monitem The first and most complex is multiplication and
mutation. By wondering what takes place in the biology, we are going to determine the fields
present in the modeling term.

As multiplication takes place in the infected lymphocytes, if there
were no such lymphocytes ($U=0$), whatever might be the number of
(free) viruses, there would be no multiplication and the term would be
zero. This would almost occur in the end of the disease. More
precisely, as once a virus has infected a lymphocyte, no more viruses
may infect it \cite{Dern}, $U$ is a good measure of the total number
of infectious viruses. Note that in our model $V$ is the number of free
and infecting viruses. So there must be a dependance on $U$ (or $U_j$ at
this stage).

As the antigenicity of an infected lymphocyte has no link with the
antigenicity of the multiplicating virus, the involved field must be
$U$ (and not $U_j$).

As there are mutations from any antigenicity $k$, the $V_k$ mutate and
their offspring is made of either $V_j$ or $W_j$ with probability
$\theta S_{kj}$ or $(1-\theta)S_{kj}$ respectively. So there must be
$\sum_k S_{kj} \theta V_k$ in the modeling term.

What can be the modeling term ? Up to now, we have $U(\sum_k
S_{kj} \theta V_k)$ times an unknown term $A(t)$. Once summing over
the antigenicity, we must find $a \theta U$. So because of
(\ref{modele_eq8.25}), $A(t)$ is such that:
\moneqstar    
\sum_j \left( \frac{\dd V_j}{\dd t} \right)_{\mathrm{growth}}= \sum_j A(t) U \left( \sum_k
S_{kj}\theta V_k \right) = a\theta U \Rightarrow A(t) = \frac{a}{V},
\monendstar
 because $\sum_j S_{kj}=1$.
So our modeling term for multiplication with mutation to $V_j$ ($V_k
\rightarrow V_j$) is:
\moneq      \label{modele_eq10}
\left( \frac{\dd V_j}{\dd t} \right)_{\mathrm{growth}}= a  \frac{U}{V} \sum_{k} S_{kj} \theta V_k.
\monend

\monitem attack of the viruses by the lymphocytes of the same antigenicity
produced by the immune system:
\moneq     \label{modele_eq11}
\left( \frac{\dd V_j}{\dd t} \right)_{\mathrm{lymphocyte}}=-\xi_j V_j T_j.
\monend
Once compiled, the evolution equation is:
\moneq  \label{modele_eq12.5}
\Frac{\dd V_j}{\dd t} = \left( \frac{\dd V_j}{\dd t} \right)_{\mathrm{growth}}+\left(
\frac{\dd V_j}{\dd t} \right)_{\mathrm{lymphocyte}}= \displaystyle \frac{aU}{V} 
\sum_k S_{kj}\theta V_k-\xi_j V_j T_j.
\end{equation} 
%

   \bigskip  \noindent {\bf   3-3-2  \quad   Non-infectious viruses}

\smallskip   
The variation of non-infectious viruses ($\dd W_j/ \dd t$) may be modeled by
various effects and translated into various terms similar to
infectious ones:

\monitem 
mutation from an antigeniticy $k$ ($V_k$ with $V_j$ included) to $W_j$:
\moneqstar    
\left( \frac{\dd W_j}{\dd t} \right)_{\mathrm{growth}}= \displaystyle\frac{a U}{V}
\sum_{k} S_{kj} (1-\theta) V_k ;
\monendstar

\monitem 
attack of the virus by the lymphocytes of the very same antigenicity. 
Notice that as the immune system may not detect whether a virus is infectious 
or not the $\xi_j$ must be the same as for $V_j$:
\moneqstar     
\left( \frac{\dd W_j}{\dd t} \right)_{\mathrm{lymphocyte}}=-\xi_j \, W_j \, T_j.
\monendstar
Once compiled the evolution equation reads:
\moneq  \label{modele_eq15}
\Frac{\dd W_j}{\dd t}=\left( \frac{\dd W_j}{\dd t} \right)_{\mathrm{growth}}+\left(
\frac{\dd W_j}{\dd t} \right)_{\mathrm{lymphocyte}}= \displaystyle \frac{aU}{V} \sum_{k}
S_{kj} (1-\theta) V_k  -\xi_j \, W_j \, T_j.
\monend
As for $V_j$, the lymphocytes do attack only the virus of the same
antigenicity. Moreover, they cannot make a distinction between
infectious or non-infectious viruses.

\bigskip     \noindent {\bf \large 3-4   \quad  The dynamical system }

\smallskip   
When we collect equations (\ref{modele_eq5}), (\ref{modele_eq6}),
(\ref{modele_eq12.5}) and (\ref{modele_eq15}), our model reads finally
as said in \cite{Dubois07}:
\begin{eqnarray}
\label{modele_eq15.1}
\Frac{\dd T_j}{\dd t}& = & - \beta_j T_j+\gamma_j+C_j(V_j+W_j) \, T_j 
-\Frac{1}{\tau_j }J\left(\Frac{V}{T} \right) T_j ,\\
\label{modele_eq15.2}
\Frac{\dd U_j}{\dd t}& = & +\Frac{1}{\tau_j }J\left(\Frac{V}{T} \right) T_j -\alpha_j \, U_j, \\
\label{modele_eq15.3}
\Frac{\dd V_j}{\dd t} & = & \displaystyle \frac{ a\theta U }{V}\sum_{k} S_{kj} V_k  -\xi_j \, V_j \, T_j, \\
\label{modele_eq15.4}
\Frac{\dd W_j}{\dd t} & = & \displaystyle \frac{ a (1-\theta) U }{V}\sum_{k} S_{kj} V_k  -
\xi_j \, W_j \, T_j.
\end{eqnarray}
Very simple manipulation enable macroscopic laws to be proven:
\begin{eqnarray}
\nonumber
\Frac{\dd (V+W)}{\dd t}& =& \displaystyle a \, U -\sum_{j}\xi_j (V_j+W_j) \, T_j; \\
\label{modele_eq16} \,\,\, \,\,\, 
\Frac{\dd \Big[    (T+U+\sum_j \big(\frac{C_j}{\xi_j}(V_j+W_j) \big) \Big]}{\dd t} 
&=&-\sum_j \beta_j T_j
+\sum_j \gamma_j -\sum_j \alpha_j U_j +aU\sum_j \frac{C_j}{\xi_j}. 
\end{eqnarray}
Since we see no reason why $\xi_j$ or $C_j$ should depend on $j$, the
equation (\ref{modele_eq16}) could be considered as a simple
linear combination of the integrated versions of
(\ref{modele_eq15.1}-\ref{modele_eq15.4}). Such a law could be
experimentally checked. 

There remains the problem of initial conditions and the way they enter
into the evolution model with mutations. Indeed, in most ordinary
differential systems, as soon as a function is identically zero in a
subdomain, it remains so (Cauchy-Picard Theorem). Moreover, $10^{-20}$
or $10^{-40}$ are numerically very different while they both mean
``zero''. In simulations of mutation, we will need a quantic jump. All
this is postponed to a forthcoming article.

The dynamical system (\ref{modele_eq15.1}-\ref{modele_eq15.4})
depends on several parameters ($\alpha_j, \beta_j, ...$). The meaning
of some of them is clear and can be easily found either in the
literature of from simple biological data. In the latter case,
parameter identification based on a good methodology enables the
numerical resoution of an inverse problem. Such studies can be found
in \cite{Bortz_Nelson,Guedj_Thiebaut_Commenges,Samson_Lavielle_Mentre}.

\bigskip \bigskip  \noindent {\bf \Large 4) \quad  Mathematical properties}
  
\smallskip 
Starting from (\ref{modele_eq15.1}-\ref{modele_eq15.4}), we make the
various fields dimensionless by using the value of
$T_{\mathrm{equil}}=\gamma_j/\beta_j$ at equilibrium as a
characteristic value. Then we let $N=1$:
\moneqstar  
T=\frac{T_j}{T_{\mathrm{equil}}}, \quad U=\frac{U_j}{T_{\mathrm{equil}}}, \quad V=
\frac{V_j}{T_{\mathrm{equil}}}, \quad W=\frac{W_j}{T_{\mathrm{equil}}}.
\monendstar  
We also define dimensionless parameters:
\moneqstar  
\omega= C_j \, T_{\mathrm{equil}}, \quad \zeta=\xi_j \, T_{\mathrm{equil}}.
\monendstar  
With  these notations, the system reads:
\begin{eqnarray}
\label{sys_DLR_1}
\dt{T}{t} & = & \beta (1-T(t)) -\Frac{T(t)}{\tau} \JVT + \omega \, \big( 
V(t)+W(t) \big) \, T(t), \\
\label{sys_DLR_2}
\dt{U}{t} & = & \Frac{T(t)}{\tau} \JVT-\alpha \, U(t), \\
\label{sys_DLR_3}
\dt{V}{t} & = & a  \,   \theta \,U(t) - \zeta \, V(t) \, T(t), \\
\label{sys_DLR_4}
\dt{W}{t} & = & a \,(1-\theta) \,U(t) - \zeta \,W(t) \, T(t).
\end{eqnarray}
The unknown fields for (\ref{sys_DLR_1}-\ref{sys_DLR_4}) are
dimensionless whereas the time variable remains dimensioned (by days).
Since we assume there is only one antigenicity ($N=1$), mutation
disappears. Hereafter, we prove the solution remains admissible and
exists globally. Then we look for the fixed points and study their
stability.

The very elegant method of \cite{Leenheer_Smith} cannot apply to our
system since it is 4D fully coupled even in its simplified form
(\ref{sys_DLR_1}-\ref{sys_DLR_4}). Moreover their results (even for their
system (8)) do not apply to our system since our infection term is not
the mass-action. In addition, they do not study models where the
immune system produces $T$ because of infection (our $C(V+W)T$) and
their system (8) is not really coupled for $W$ like ours.

\bigskip     \noindent {\bf \large 4-1    \quad  The solution remains admissible}

\smallskip  
We intend to prove that the system (\ref{sys_DLR_1}-\ref{sys_DLR_4})
is mathematically well posed and biologically meaningful: it
does not exhibit negative values of the fields for
admissible initial conditions.
To that purpose, we make various assumptions on the parameters:
\moneq  \label{assumptions_DLR_1}
\left\{     \begin{array}{l} 
\beta, \tau, \omega, \alpha, \zeta \mbox{ are real positive and } 0< \theta < 1,\\
J(\smb) \mbox{ is a real function concave over } [0, \infty[, J(0)=0, J'(0)=1,\\
J(x) \rightarrow 1 \mbox{ when } x \rightarrow +\infty \,\, 
 \mbox{ and } \, J   \mbox{ is bounded on } \mathbb{R}^-.
\end{array} \right. \monend  
>From the biological meaning of the system, we define the set of
admissible fields:
\moneq  \label{Feq5}
\mathcal{B} = \{ (T,U,V,W), \,T>0, \,U \geq 0, \,V \geq 0, \,W \geq 0 \}, 
\monend 
and its interior:
\moneq  \label{Feq8}
\mathring{\mathcal{B}} = \big\{ (T,U,V,W), \,T>0, \,U > 0, V\, > 0,\, W > 0 \big\}.
\monend 
Due to the Cauchy-Picard theorem, we know that there exists
a local in time solution. The question is then whether this local
solution is admissible. Our main result is the following Theorem.

\begin{theorem}[Biological consistance]
\label{Fth1}
The solution $(T(t),U(t),V(t),W(t))$ of system
(\ref{sys_DLR_1}-\ref{sys_DLR_4}) with an initial condition in
$\mathcal{B}$ remains in $\mathcal{B}$.
\end{theorem}

\noindent 
To prove Theorem \ref{Fth1}, we will need various lemmas. The first one
states that the number of lymphocytes does not vanish in finite
time.

\begin{lemma}
\label{Fth2}
If the initial condition $(T_0,U_0,V_0,W_0)$ is in $\mathcal{B}$, 
then $T(t)>0$ for all time for which the fields $T,U,V,W$ are defined.
\end{lemma}

\noindent $\bullet$ \quad  Proof of Lemma 8.  

\noindent 
Thanks to the Cauchy-Picard theorem \cite{Arnold}, we know that for
$t$ sufficiently small, $T(t)$ is positive. So if there exists at
least one time $\tilde{t}$ such that $T(\tilde{t})=0$, then we define
$t^*$ to be the smallest and we have $t^*>0$. Since $J(\smb)$ is bounded
(whatever $V$ and $T$), (\ref{sys_DLR_1}) writes at that time
\moneqstar  
\dt{T}{t}(t^*)= \beta >0.
\monendstar  

 \noindent 
As a consequence, for $t$ below and sufficiently close to $t^*$,
$T(t)<0$. But as $T_0>0$, from the intermediate value theorem, there
exists a $t'$ smaller than $t^*$ for which $T(t')=0$. This contradicts
the assumption that $t^*$ is the smallest and completes the proof.
\hfill $ \square $

\bigskip  \noindent 
We need now to study the various cases where the initial conditions
are either in $\mathring{\mathcal{B}}$ or on the boundary of
$\mathcal{B}$. This will be discussed through some lemmas where the
initial condition has either zero, three, two or one initial vanishing
fields.
In the case of initial condition in the interior of $\mathcal{B}$, one
may state the following Lemma.

\begin{lemma}[Zero vanishing initial condition]
\label{Fth4}
If the initial condition of system
(\ref{sys_DLR_1}-\ref{sys_DLR_4}) is in $\mathring{\mathcal{B}}$, then
the solution remains in $\mathring{\mathcal{B}}$ for any $t\geq 0$
provided it exists.
\end{lemma}

\noindent $\bullet$ \quad  Proof of Lemma 9.

\noindent 
We will discuss the cases where one, two or three fields vanish
simultaneously.

\monitem
Let us assume $W$ vanishes first and alone (before the other fields) and then
let us denote $t^*$ the smallest time for which $W$ vanishes. On
$[0,t^*]$, one has:
\moneq  
V(t) >0, \,\,U(t) >0, \,\,W(t)\geq 0, \,\,W(t^*)=0,
\monend  
in addition to the fact that $T(t)>0$ (see Lemma \ref{Fth2}). The
equation (\ref{sys_DLR_4}) writes 
$ \frac{\dd W}{\dd t} =  a \, (1-\theta) \, U(t^*) >0$
because $0< \theta <1$ (see the assumption
(\ref{assumptions_DLR_1})). So, there exists a time $\tilde{t}$ smaller
than $t^*$ at which $W(\tilde{t})<0$. From the intermediate
value theorem, one may conclude that there exists a time smaller than
$t^*$ (and than $\tilde{t}$) at which $W$ vanishes. This contradicts
the definition of $t^*$ and so $W$ may not vanish first.
 
Identical arguments enable to prove that $V$ may not vanish first.
To prove that $U$ may not vanish first neither, we assume that $t^*$
is the first vanishing time of $U$. So, on $[0,t^*]$:
\moneq  \label{Feq30}
U(t)\geq 0, \,\, U(t^*)=0, \,\,  V(t)>0,\,\,  W(t)>0.
\monend  
Thanks to Lemma \ref{Fth2}, one has $T(t^*)>0$ and so
$ \frac{\dd U}{\dd t} (t^*)=J( \frac{V}{T} ) \,  \frac{T(t^*)}{\tau}  >0$. 
In a similar way to the two
previous cases, one gets a time $\tilde{t}$ smaller than $t^*$ for
which $U(\tilde{t})<0$. Thanks to the intermediate value theorem, one
gets also a time smaller than $t^*$ where $U$ vanishes. This
contradicts the definition of $t^*$. So $U$ may not vanish first
neither.

\monitem
Let us discuss now the case where two fields vanish simultaneously.
The case where $V$ and $W$ vanish simultaneously and alone is
impossible for the same arguments as the case where $W$ vanishes
first.
The case where $U$ and $W$ vanish simultaneously and alone may be
treated in the same way as the case of $U$ vanishing first.
The only remaining case is if $U$ and $V$ vanish simultaneously (and
not $W$). Let us then define $t^*$ the smallest such vanishing
time. The Cauchy-Picard theorem, applied to the system
(\ref{sys_DLR_1}-\ref{sys_DLR_4}) with reversed time, enables us to
claim that the (unique) solution is also such that
\moneqstar  
\dt{T}{t} =  \beta (1-T(t)) + \omega W(t)T(t), \quad
\dt{W}{t} = - \zeta W(t) T(t),  \quad   U(t) = V(t) = 0, 
\monendstar
for all time $t$ in $[t^* - \epsilon, \,t^*]$ for some $\epsilon > 0.$
 This contradicts the definition of $t^*$ as the smallest vanishing time.  

\monitem In the case where $U,V$ and $W$ vanish at the same time, the
proof is the same as for Lemma \ref{Fprop1}.
\hfill $ \square $

\bigskip  \noindent 
The following Lemma solves the case of three vanishing initial fields.

\begin{lemma}[Three vanishing initial conditions]
\label{Fprop1}
If the initial condition is $T_0>0, U_0=V_0=W_0=0$, then the solution is unique and is health.
\end{lemma}

\noindent $\bullet$ \quad  Proof of Lemma 10.

\noindent 
The proof relies only on the Cauchy-Picard theorem which states that
\moneqstar   
T(t)=1-(1-T_0)\exp{(-\beta t)}, 
\monendstar
and the other fields identically zero.
\hfill $ \square $

\bigskip  \noindent  
The next Lemma deals with the case where two initial fields vanish.

\begin{lemma}[Two vanishing initial conditions]
\label{Fth3}
If the initial condition is among
\begin{eqnarray}
\label{Feq10}
 & & T_0>0, \,\, U_0=V_0=0,  \,\, W_0>0, \\
\label{Feq25}
 & & T_0>0,  \,\, U_0>0,  \,\,  V_0 = W_0 = 0, \\
\label{Feq26}
 & & T_0>0,  \,\,U_0=0,  \,\, V_0>0,  \,\, W_0=0,
\end{eqnarray}
the solution of system (\ref{sys_DLR_1}-\ref{sys_DLR_4}) 
remains in $\mathcal{B}$.
\end{lemma}

\noindent $\bullet$ \quad  Proof of Lemma 11.

\noindent The proof distinguishes various cases.

\monitem
In the case (\ref{Feq10}), the Cauchy-Picard theorem enables us to
claim that there exists a solution of
(\ref{sys_DLR_1}-\ref{sys_DLR_4}) for which $U(t)=V(t)=0$, and the
fields $T(t),W(t)$ satisfy
\begin{eqnarray}
\label{Feq11}
\dt{T}{t} & = & \beta (1-T(t)) + \omega \, W(t)\,T(t) \\
\label{Feq12}
\dt{W}{t} & = & - \zeta \,W(t) \,T(t).
\end{eqnarray}
Let us assume $W$ vanishes at some times. Among these times, we
chose $t^*$ to be the smallest. On the compact set $[0,t^*]$, $T$ is
continue and so bounded by $\gamma >0$. So (\ref{Feq12}) enables to claim
\moneqstar   
\dt{W}{t}  \geq - \gamma \zeta W  \quad {\rm then }  \quad
 W(t) \geq W_0 \exp{(- \zeta \gamma t)}  \quad {\rm for } \,\, t \geq 0
\monendstar 
which contradicts the assumption that $W(t^*)=0$. So $W$ may not
vanish in finite time and so in the case (\ref{Feq10}), $W(t)>0$ for
any positive time and $U=V=0$.

\monitem
In the case (\ref{Feq25}), $\frac{\dd V}{\dd t} > 0$ 
and so $V(t)>0$ for $t$
small enough. Similarly, for $t$ small enough, $W(t)>0$. So for $t$
small enough, the solution enters the interior of $\mathcal{B}$. Such
a new initial condition has been adressed in Lemma \ref{Fth4}.
So the solution remains in $\mathcal{B}$.

\monitem
In the case (\ref{Feq26}), 
$\frac{\dd U}{\dd t} > 0$ thanks to
(\ref{sys_DLR_2}), and so $U(t)>0$ for $t$ sufficiently small. Indeed,
$\frac{\dd W}{\dd t} = 0$
but  $\frac{\dd^2 W}{\dd t^2} =  a \, (1-\theta)  \, \frac{\dd U}{\dd t} > 0 .$ 
This is
enough to assess that $W(t)>0$ for $t$ sufficiently small. Using the
new ``initial'' condition at this time, we are driven back
to the case treated by Lemma \ref{Fth4}.
This completes the proof.
\hfill $ \square $

\bigskip  \noindent  

\begin{lemma}[One vanishing initial condition]
\label{Fthpouet}
If one and only one initial field vanishes, the solution remains 
in $\mathcal{B}$.
\end{lemma}

\noindent $\bullet$ \quad  Proof of Lemma 12.

\noindent If $U_0=0$ (and $V_0 \, W_0 >0$), then  $\frac{\dd U}{\dd t} (0) = \frac{1}{\tau}
 T \, J(\frac{V}{T}) > 0$. 
One may conclude in a similar way to the case
(\ref{Feq26}) treated in Lemma \ref{Fth3}.
If $V_0=0$ (and $U_0 \, W_0 >0$), then  $\frac{\dd V}{\dd t} (0) = a  \, \theta \, U_0>0$. 
So in finite time, one is driven back to the case treated
in Lemma \ref{Fth4} (zero vanishing intitial condition).
If $W_0=0$ (and $U_0 \, V_0 >0$), the argument is very similar.
\hfill $ \square $

\bigskip  \noindent  

\noindent $\bullet$ \quad  Proof of Theorem 7. 

\noindent 
Up to now, we have proved that if the initial condition does not vanish 
(Lemma \ref{Fth4}), vanishes three times (Lemma \ref{Fprop1}), 
two times (Lemma \ref{Fth3}), or once (Lemma \ref{Fthpouet}), 
the solution remains in $\mathcal{B}$. 
This completes the proof of Theorem \ref{Fth1}.
\hfill $ \square $

\bigskip   \newpage   \noindent {\bf \large 4-2    \quad  Global existence}

\noindent 
The following Theorem states that the solution remains finite and so
is global in $t$.

\begin{theorem}
\label{Fth5}
Let 
\moneq  \label{eta}
\eta \,=\, a \, \omega  -  \alpha \, \zeta,
\monend
and 
\moneq   \label{gamma}
\gamma \,=\,  \Bigl|   \frac{\eta}{\zeta}   \Bigr|  \, . 
\monend
If the initial condition
$(T_0, \,U_0, \,V_0, \,W_0)$ is in $\mathcal{B}$, then the solution of
(\ref{sys_DLR_1}-\ref{sys_DLR_4}) satisfies:
\moneq  \label{casfort}
T(t)+U(t)+\Frac{\omega}{\zeta}(V(t)+W(t))  \, \leq  \, 
  T_0+U_0+ \Frac{\omega}{\zeta}(V_0+W_0)  \, + \, \beta \, t \qquad {\rm if } 
\quad  \eta \leq  0 ,
\monend
 and 
\moneq  \label{Feq34}
T(t)+U(t)+\Frac{\omega}{\zeta}(V(t)+W(t))  \, \leq  \, 
 \big( T_0+U_0+ \Frac{\omega}{\zeta}(V_0+W_0) \big) 
\, {\rm e}^{\displaystyle  \gamma t} 
 \,+\,  \frac{\beta}{\gamma} \, 
\big( {\rm e}^{\displaystyle  \gamma t}   -1 \big)  
\quad {\rm if }  \,  \eta >  0 . 
\monend
As a consequence, the solution is finite for all time
$t\in [0,+\infty[$ and so is global.
\end{theorem}

\noindent $\bullet$ \quad  Proof of Theorem 13. 

\noindent By adding (\ref{sys_DLR_1}-\ref{sys_DLR_2}) and $
\frac{\omega}{\zeta} $ times the sum of (\ref{sys_DLR_3}) and
(\ref{sys_DLR_4}), all the non-linear terms disappear and one has:
\moneq  \label{Feq35}
\dt{T}{t}+\beta T+\dt{U}{t} -\frac{\eta}{\zeta} \, U +
 \Frac{\omega}{\zeta} \, \dt{}{t} (V+W) = \beta.
\monend
Thanks to Theorem \ref{Fth1}, the solution remains in
$\mathcal{B}$ and so $ U\geq 0 $. 
Moreover if   $ \, \eta \leq 0 $, 
we can minorate the fourth term of  (\ref{Feq35})  by $0$
because $ \zeta > 0 $. As a consequence, 
\moneqstar   
\dt{}{t} \Big(   T+U+\Frac{\omega}{\zeta}(V+W) \Big) \, \leq \, \beta, 
\monendstar
and the relation  (\ref{casfort}) is a simple consequence of the integration in time of
the previous inequality.

\noindent 
If   $ \, \eta > 0 $, we integrate 
 between $0$ and $t$ the equation (\ref{Feq35}), 
and owing to the positivity of $T$, one has
\moneq  \label{Feq36}
T+U+\Frac{\omega}{\zeta}(V+W) \leq 
T_0+U_0+\Frac{\omega}{\zeta}(V_0+W_0) + \beta t + \gamma \int_0^t U(t') \, \dd t'.
\monend
Let us denote 
\moneqstar  
\delta_0 \equiv T_0+U_0+\Frac{\omega}{\zeta}(V_0+W_0) \quad {\rm and} \quad 
\phi(t) \equiv \int_0^t U(t') \, \dd t' .   
\monendstar
Due to Theorem \ref{Fth1} we know that the solution
remains in $\mathcal{B}$ ($T>0, V\geq 0, W\geq0$), the
equation (\ref{Feq36}) enables to write:
\begin{eqnarray}
\Frac{\dd \phi}{\dd t}  & \leq &  \delta_0 + \beta t + \gamma \phi(t), \nonumber \\
 \label{Feq39} 
\Frac{\dd}{\dd t} (\exp{(-\gamma t)} \phi)  & \leq &   \exp{(-\gamma t)} \, 
(\delta_0 + \beta t).
\end{eqnarray} 
The  inequality  (\ref{Feq39}) can be integrated ($\phi(0)=0$):
\moneqstar    
\exp{(-\gamma t)}\phi(t) \leq \left[-\Frac{\beta}{\gamma}t
-\Frac{1}{\gamma}\left(\delta_0+\Frac{\beta}{\gamma}\right) \right]
\exp{(-\gamma t)}+ \Frac{1}{\gamma} \left(\delta_0+\Frac{\beta}{\gamma}\right),
\monendstar
or
\moneqstar   
\gamma \phi(t) \, \leq \, - \left( \beta t + \delta_0 +  \frac{\beta}{\gamma} \right)
 \,+\, \left( \delta_0   \,+\,  \frac{\beta}{\gamma}  \right)\,  \exp{(+\gamma t)}.
\monendstar
With such a bound for $\phi$, one may take back the right hand side of (\ref{Feq36}):
\moneqstar  
T+U+\Frac{\omega}{\zeta}(V+W) 
\, \leq  \,  \delta_0 +\beta t + \gamma \phi  
\, \leq  \, \delta_0 \, \exp{(\gamma t)}  +  
 \frac{\beta}{\gamma} \,  (\exp{(\gamma t)}-1),
\monendstar
which is the inequality (\ref{Feq34}). This completes the proof.
\hfill $ \square $ 

\bigskip     \noindent {\bf \large 4-3    \quad  Fixed points}

\smallskip   
Looking for fixed points of (\ref{sys_DLR_1}-\ref{sys_DLR_4}), we must
find solutions of the associated stationary system:
\moneq  \label{fix_DLR_1}
\begin{array}{rcl}
0 & = & \beta (1-T) -\Frac{T}{\tau} J\left( \Frac{V}{T} \right) + \omega \, \big( 
V+W \big) \, T, \\
\Frac{T}{\tau}\,  J\left( \Frac{V}{T} \right)  & = & \alpha \, U, \\
a  \,   \theta \,U & = & \zeta \, V \, T, \\
a \,(1-\theta) \,U & = & \zeta \,W \, T.
\end{array} \monend  
We will prove the following property:
\begin{proposition}
\label{prop14}
Let $\eta$ defined in (\ref{eta}) and
\moneq  \label{HP0}
\begin{array}{rcl}
\rho & = & \Frac{\alpha \zeta \tau }{a \theta},\\
\overline{V} & = & \Frac{\beta a \theta }{a \omega -\alpha \zeta}.
\end{array} \monend 
The fixed points of (\ref{sys_DLR_1}-\ref{sys_DLR_4}) 
depend on the sign of $\eta$. Three cases must be distinguished. 

\monitem If $\eta > 0$, the fixed points are either health and
seropositivity (if $\rho <1$), or only health (if $\rho \geq 1$).

\monitem If $\eta=0$, there is no fixed point else than health.

\monitem If $\eta <0$, health is always a solution. Moreover, there
appears a non-explicit threshold value $L$ and three sub-cases
depending on $L$: 

\sonitem When $\rho<1$ there is one seropositivity fixed point. 

\sonitem When $1 < \rho < L$, there are two seropositivity fixed points. 

\sonitem When $\rho > L$ there is no seropositivity solution.  
\end{proposition}

\noindent $\bullet$ \quad  Proof of Proposition 14.

\noindent One finds easily that 
\moneq  \label{fix_DLR_2}
\begin{array}{c}
(1-\theta) V= \theta W, \\
U = \Frac{\zeta}{a} \, (V +W) \, T=\Frac{\zeta}{a \theta} \, V \, T .
\end{array} \monend  
Then, the system (\ref{fix_DLR_1}) reduces to
\begin{eqnarray}
\label{fix_DLR_3}
0 & = & \beta (1-T) -\Frac{\alpha \zeta}{a \theta}VT + \Frac{\omega}{\theta} VT,  \\
\label{fix_DLR_4}
\Frac{T}{\tau} J\left( \Frac{V}{T} \right)& = & \Frac{\alpha \zeta}{a \theta} \, V \, T.
\end{eqnarray}
Since $T=0$ is not a solution, one may simplify $T$ in (\ref{fix_DLR_4}). Equation
(\ref{fix_DLR_3}) gives $T$:
\moneq  \label{fix_DLR_5}
T = \Frac{\beta}{\beta -\frac{(\omega a -\alpha \zeta)}{a \theta} V}=\Frac{1}{1-V/\overline{V}}.
 \monend  
Thanks to the value of $T$ given by (\ref{fix_DLR_5}), we are driven
to solve (\ref{fix_DLR_4}) in the form:
\moneq  \label{fix_DLR_7}
J\left(V\Frac{(\overline{V} -V)}{\overline{V}}\right) = 
\Frac{\alpha \zeta \tau }{a \theta} V = \rho V. 
\monend   
Three cases appear to solve this equation.

\monitem If $\eta >0$, we must find the solution $V$ of (\ref{fix_DLR_7}) where $
\overline{V}>0$. To guarantee $T\geq 0$, we must have $V <\overline{V}$.
Since the function $V\mapsto
J(V(\overline{V}-V)/\overline{V})$ is symmetric with respect to
$\overline{V} /2$ and the right hand side $\rho V$ is a linear function of $V$,
only two subcases must be distinguished according to the Figure~6. 

\bigskip  \centerline {
{\includegraphics[width=.49 \textwidth]{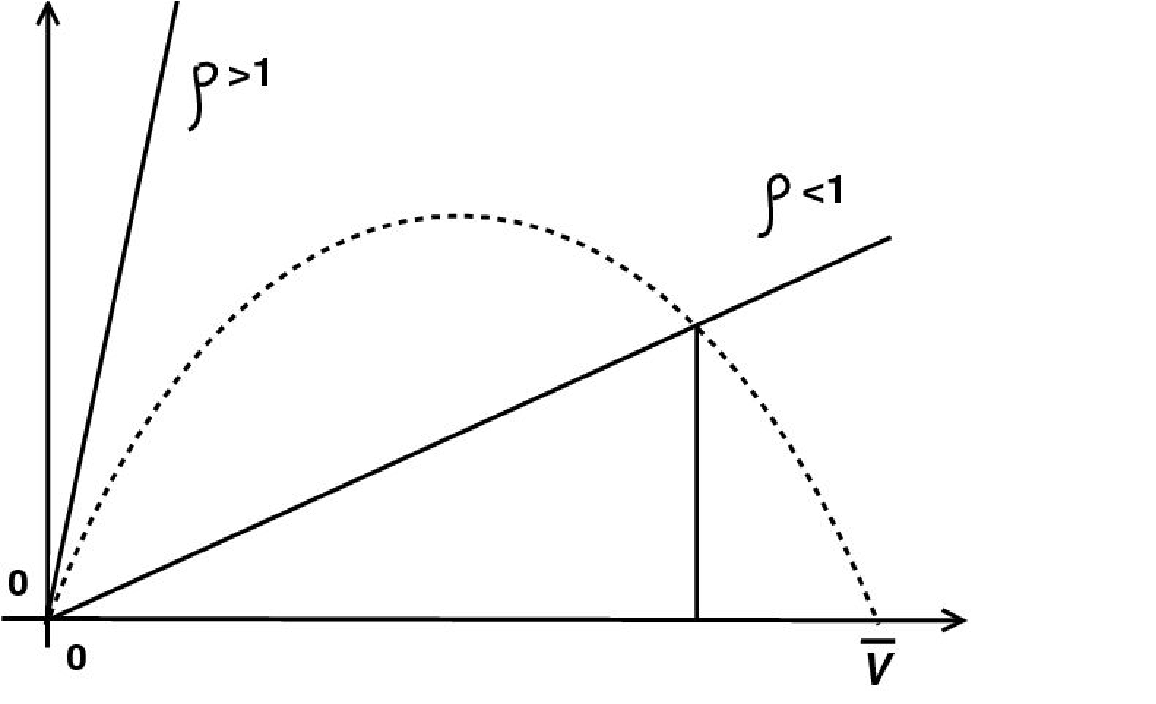}}  }  
\hangindent=7mm \hangafter=1 \noindent {\bf Figure 6}. $\,$   
{\it   Two characteristic shapes of the curves $V\mapsto \rho V$ and $V\mapsto
J(V(\overline{V}-V)/\overline{V})$ (dotted) for $\eta >0$. }
\bigskip  

\noindent This discussion is summarized in the following:

\noindent 1. \quad If $\rho <1$ there are two solutions: 

\quad {\em (i)} \quad $V_1=0 \Rightarrow T_1=1, U_1=0= W_1$ ; 

\quad {\em (ii)} \quad $V_2 >0 \Rightarrow T_2>1$. 

\noindent The first one will be denoted health and the second one
seropositivity. We notice that $V_2 < \overline{V}$ and in the end,
the condition on $V$ to ensure $T>0$ is satisfied.
 
\noindent 2. \quad If $\rho \geq 1$, the only solution is health. \\ 

\monitem If $\eta =0$, then $\overline{V}=\infty$. There is no solution.

\monitem If $\eta <0$, then $\overline{V} <0$ and there is no constraint
  on $V$ to ensure that $T \geq 0$. Unlike the case $\eta >0$, the
  parabola inside the function $J$ is of the type $y=+x^2$ instead of
  $y=-x^2$. So as to circumvent this, we will invert the function $V
  \mapsto V(1-V/ \overline{V})$ on $\mathbb{R}^+$ where it is
  one-to-one. For any $V\in \mathbb{R}^+$ (we look for non-negative $V$),
  there is a unique $X \in
  \mathbb{R}^+$ such that
\moneq  \label{HP1}
X=V(1-V/\overline{V}) \Leftrightarrow V=\psi(X) = \Frac{\overline{V} +
  \sqrt{\overline{V}^2-4 \overline{V} X}}{2}.
\monend  
%

\bigskip  \centerline {
{\includegraphics[width=.49 \textwidth]{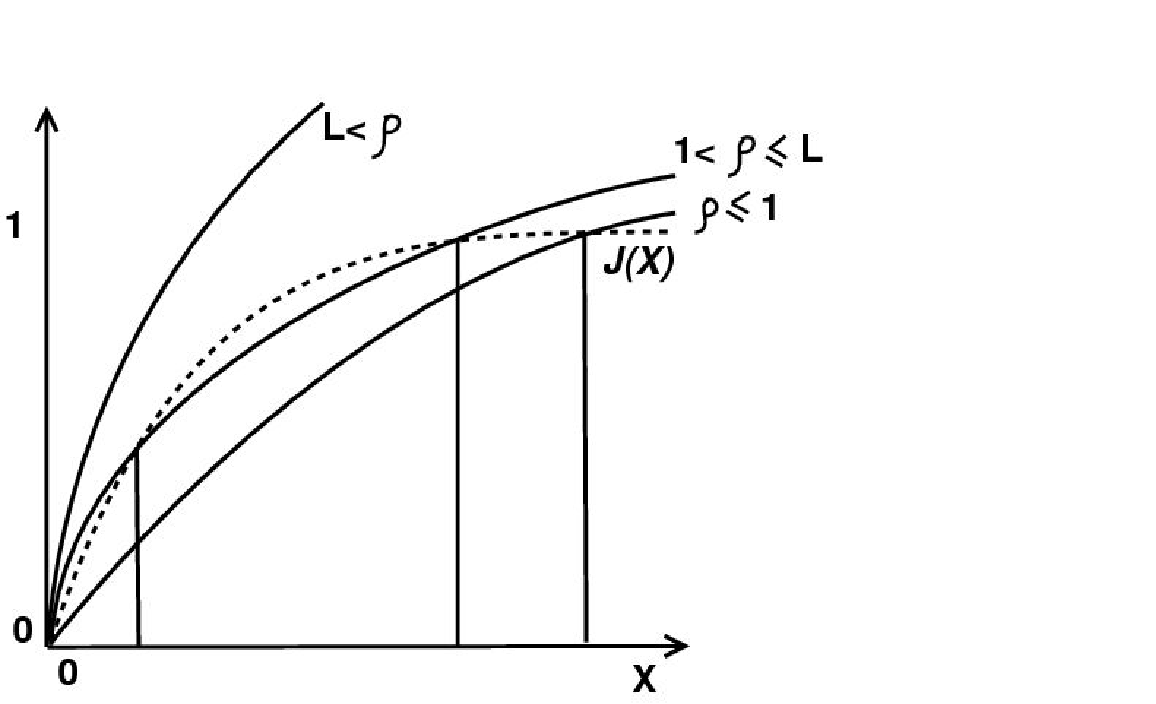}}  }  
\hangindent=7mm \hangafter=1 \noindent {\bf Figure 7}. $\,$   
{\it  Three characteristic shapes of the curves $X \mapsto \rho \psi(X)$ and $J$ (dotted) with
$J \equiv $ tanh  and $\eta <0$. }
\bigskip   \bigskip              

\noindent 
Then equation (\ref{fix_DLR_7}) amounts to
\moneq  \label{HP2}
J(X)=\rho \psi(X).
\monend   
\noindent 1. \quad If $\rho \leq 1$, there are two solutions denoted $X_i^1$: 

\quad ($i$) \quad $X_1^1=0$ corresponding to health ; 

\quad ($ii$) \quad $X_2^1>0$ corresponding to seropositivity. 

\noindent 2. \quad If $1 < \rho \leq L$ there are three solutions denoted by
$X_i^2$: 
 
\quad ($i$) \quad $X_1^2=0$ corresponding to health ;
 
\quad ($ii$) \quad $X_2^2 > 0$ corresponding to seropositivity ;
 
\quad ($iii$) \quad $0 < X_3^2 < X_2^2$ corresponding to seropositivity. 

\noindent Notice that there is no explicit value of $L$. Yet for $J \equiv $ tanh and $\overline{V}
  \simeq -0.054$, we could numerically find $L\simeq 3.7$.
 
\noindent 3. \quad If $\rho > L$, there is only one solution $X_1^3 =0$ (health). 
\hfill $ \square $ 

\bigskip     \noindent {\bf \large 4-4    \quad  Stability} 

\smallskip   
The Jacobian matrix is, at every fixed point:
\moneqstar 
\left(\begin{array}{cccc}
\Frac{V}{T \tau} J'\left(\Frac{V}{T}\right) -\Frac{\beta}{T} & 0 & \omega T
-J'\left(\Frac{V}{T}\right)/\tau & \omega T \\
J\left(\Frac{V}{T}\right)/\tau -\Frac{V}{T \tau} J'\left(\Frac{V}{T}\right) &
-\alpha & J'\left(\Frac{V}{T}\right)/\tau & 0 \\
-\zeta V & a \theta & -\zeta T & 0 \\
-\zeta W & a (1-\theta) & 0 & -\zeta T
\end{array} \right). \monendstar 
If we compute the characteristic polynomial through the last right
column, we find:
\begin{equation}  \label{HP3}  \left\{
\begin{array}{c}
P(\lambda)= -(\zeta T + \lambda)\left| \begin{array}{ccc}
\Frac{V}{T \tau} J'\left(\Frac{V}{T}\right) -\Frac{\beta}{T} -\lambda & 0 & \omega T
-J'\left(\Frac{V}{T}\right) / \tau \\
J\left(\Frac{V}{T}\right)/\tau -\Frac{V}{T \tau} J'\left(\Frac{V}{T}\right) &
-\alpha -\lambda & J'\left(\Frac{V}{T}\right)/\tau \\
-\zeta V & a \theta & -\zeta T -\lambda
\end{array} \right| \\  ~ \\
-\omega T
\left| \begin{array}{ccc}
J\left(\Frac{V}{T}\right)/\tau -\Frac{V}{T \tau} J'\left(\Frac{V}{T}\right) &
-\alpha -\lambda & J'\left(\Frac{V}{T}\right)/\tau \\
-\zeta V & a \theta & -\zeta T - \lambda\\
-\zeta W & a (1-\theta) & 0 
\end{array} \right|.
\end{array}  \right.    \end{equation}  
Thanks to (\ref{fix_DLR_2}), we can easily simplify the second line of the
second determinant of (\ref{HP3}). As a consequence one may write:
\moneqstar  
P(\lambda)= -(\zeta T + \lambda)\left| \begin{array}{ccc}
\Frac{V}{T \tau} J'\left(\Frac{V}{T}\right) -\Frac{\beta}{T} -\lambda & 0 & \omega T
-J'\left(\Frac{V}{T}\right) / \tau \\
J\left(\Frac{V}{T}\right)/\tau -\Frac{V}{T \tau} J'\left(\Frac{V}{T}\right) &
-\alpha -\lambda & J'\left(\Frac{V}{T}\right)/\tau \\
-\zeta V & a \theta & -\zeta T -\lambda
\end{array} \right|.  \monendstar  
\moneqstar  -\omega T (\zeta T +\lambda)
\left| \begin{array}{cc}
J\left(\Frac{V}{T}\right)/\tau -\Frac{V}{T \tau} J'\left(\Frac{V}{T}\right) &
-\alpha -\lambda \\
-\zeta W & a (1-\theta)
\end{array} \right|.  \monendstar  
In the general case, no more factorization could be found and the
polynomial is
\moneq \label{HP4}  \!\! \!\! \!\! \!\! \left\{  \begin{array}{c}
P(\lambda ) =  -(\zeta T + \lambda ) \left[ \left( \Frac{V}{T \tau}
J'\left(\Frac{V}{T}\right) -\Frac{\beta}{T} -\lambda \right)\left( 
-{\displaystyle {{a \theta}\over{\tau}}}
J'\left(\Frac{V}{T}\right)  + (\alpha +\lambda ) (\zeta T
+\lambda)\right) \right. \\
+\left. \left(\Frac{\omega T}{\theta}
-  {\displaystyle {{1}\over{\tau}}}  J'\left(\Frac{V}{T}\right)  \right)\left( 
{\displaystyle {{a \theta}\over{\tau}}}
\left(J\left(\Frac{V}{T}\right)  -\Frac{V}{T \tau}
J'\left(\Frac{V}{T}\right)\right)-(\alpha +\lambda )\zeta V\right)\right].
\end{array}  \right.    \monend

   \bigskip  \noindent {\bf   4-4-1  \quad    Stability of the health state}

In the case of health ($T=1,U=0=V=W$), the eigenvalues are zeros of
(\ref{HP4}) which can be rewritten thanks to (\ref{HP0}):
\moneq  \label{HP5}
P(\lambda)=(\zeta T +\lambda)(\beta +\lambda)\left( \Frac{a
\theta}{\tau }(-1 +\rho)+\lambda (\alpha + \zeta) 
+\lambda^2 \right) = 0. 
\monend   
We will prove the following theorem.
\begin{theorem}
\label{HP6}
If $\rho > 1$, health is (locally) stable.\\
If $\rho < 1$, there exists one positive eigenvalue associated to an
admissible eigenvector in the sense of Definition \ref{admiss}.
\end{theorem}
%

\noindent $\bullet$ \quad  Proof of Theorem 15.

\noindent The discriminant of the non-reduced second order
polynomial in (\ref{HP5}) 
is $(\alpha-\zeta)^2+4a\theta/\tau >0$. This enables us to claim that:
%

\monitem if $\rho > 1$, the four roots are negative and so the evolution is (locally) stable;

\monitem if $\rho < 1$, there exists one (and only one) positive
root and so the evolution is locally unstable in one direction. 
We still have to prove that the
eigenvector $\vec{v}$ associated to a positive eigenvalue of
(\ref{HP5}) is admissible. In other words, we need to prove that
$\vec{v}$ is such that for $\varepsilon >0$ or $\varepsilon < 0$ small
enough, $(1,0,0,0) + \varepsilon \vec{v}$ has its four components
non-negative. To prove this, we write the system satisfied by the
eigenvector:
\moneq \label{HP7}  \left(
\begin{array}{cccc}
-(\beta +\lambda) & 0 & \omega -1/\tau & \omega \\
0 & -(\alpha  +\lambda) & 1/\tau & 0 \\
0 & a \theta & -(\zeta  +\lambda) & 0 \\
0 & a(1-\theta) & 0 & -(\zeta  +\lambda)
\end{array}
\right)
\left(\begin{array}{c}
x_1 \\ x_2 \\  x_3 \\ x_4
\end{array} \right)
 =
\left(\begin{array}{c}
0 \\ 0 \\  0 \\ 0
\end{array} \right),  \monend 
where $\rho <1$ and $\lambda$ is the (unique) positive root of
(\ref{HP5}). More precisely, $\lambda$ is the unique root of the third
term: $a \theta (-1+ \rho) /\tau +\lambda (\alpha +\zeta )
+\lambda^2$. As this second order polynomial is the determinant of the
$2 \times 2 $ submatrix in the center of the matrix in (\ref{HP7}),
and because of the very particular shape of the lines $2$ and $3$, we
can claim these lines are bound. It suffices then to take out the third line
to be driven to the system equivalent to (\ref{HP7}):
\moneqstar  
\left\{
\begin{array}{rcl}
(\beta +\lambda) x_1 & = & (\omega -1/\tau) x_3 + \omega x_4 \\
(\alpha +\lambda) x_2 & = & x_3 / \tau \\
a(1-\theta) x_2 & = & (\zeta +\lambda) x_4.
\end{array} \right. \monendstar 
As $\lambda >0$, we can see that there exists solutions such that
$x_2, x_3, x_4$ be non-negative. So, at least locally the solution in the
direction of this eigenvector is admissible and the proof is complete.
\hfill $ \square $ 

   \bigskip   \bigskip   \noindent {\bf   4-4-2  \quad   Stability of seropositivity}

  \smallskip   
We have no rigorous study of the stability/unstability of
seropositivity. So we use numerical simulations. In this
subsection, we take 
\moneqstar   
\beta = 0.01\mbox{ day}^{-1}, 
\alpha = 0.7\mbox{ day}^{-1}, \omega = 0.01\mbox{ day}^{-1}, a= 250\mbox{ day}^{-1},
\theta = 0.1, 
\monendstar 
and initial values 
\moneqstar    
T(0)=1, \quad U(0)=0, \quad V(0)=W(0)=0.05.
\monendstar 
The other parameters ($\tau, \zeta$) are
taken so as to illustrate the fixed points depicted in
Proposition \ref{prop14}. 
We provide the evolution on a short time and a phase portrait for each
case. Let us remind the fixed points:

 \smallskip   \noindent 1. \quad  $\eta >0$. 

\quad {\em (i)} \quad If $\rho <1$: health (partially unstable) and one
  seropositivity numerically stable as can be seen on Figure~8 
In this simulation, $\eta=1.8$, $\rho=0.28$.
The effect of the infection is to {\em emphasize} the activity of the 
immune system ;
 
\quad {\em (ii)} \quad If $\rho \geq 1$: only health (stable), as in 
 Figure~9 
 ($\eta=0.4, \rho=1.68$). Minimum
value for lymphocytes is obtained for 
$t \simeq 0.8$ day and maximal one for $t \simeq 16.4$ days.  

\bigskip  \newpage  \centerline {
{\includegraphics[width=.49 \textwidth]{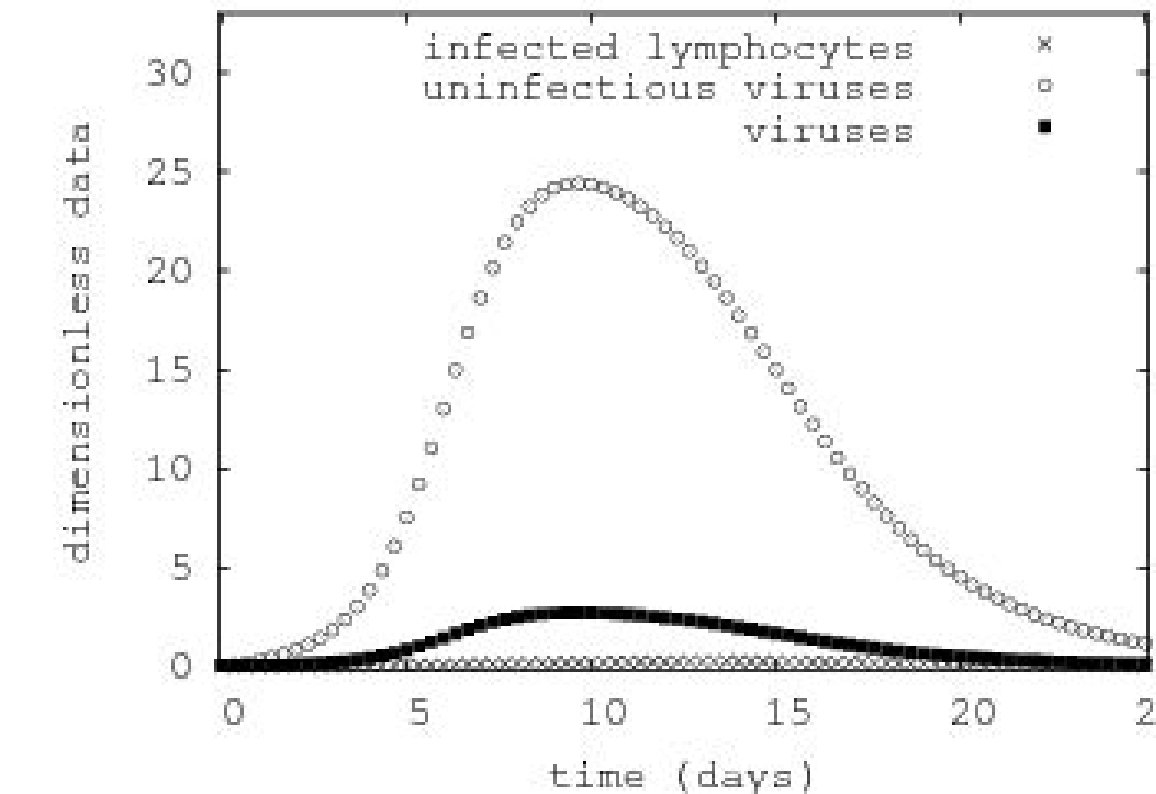}}  
{\includegraphics[width=.49 \textwidth]{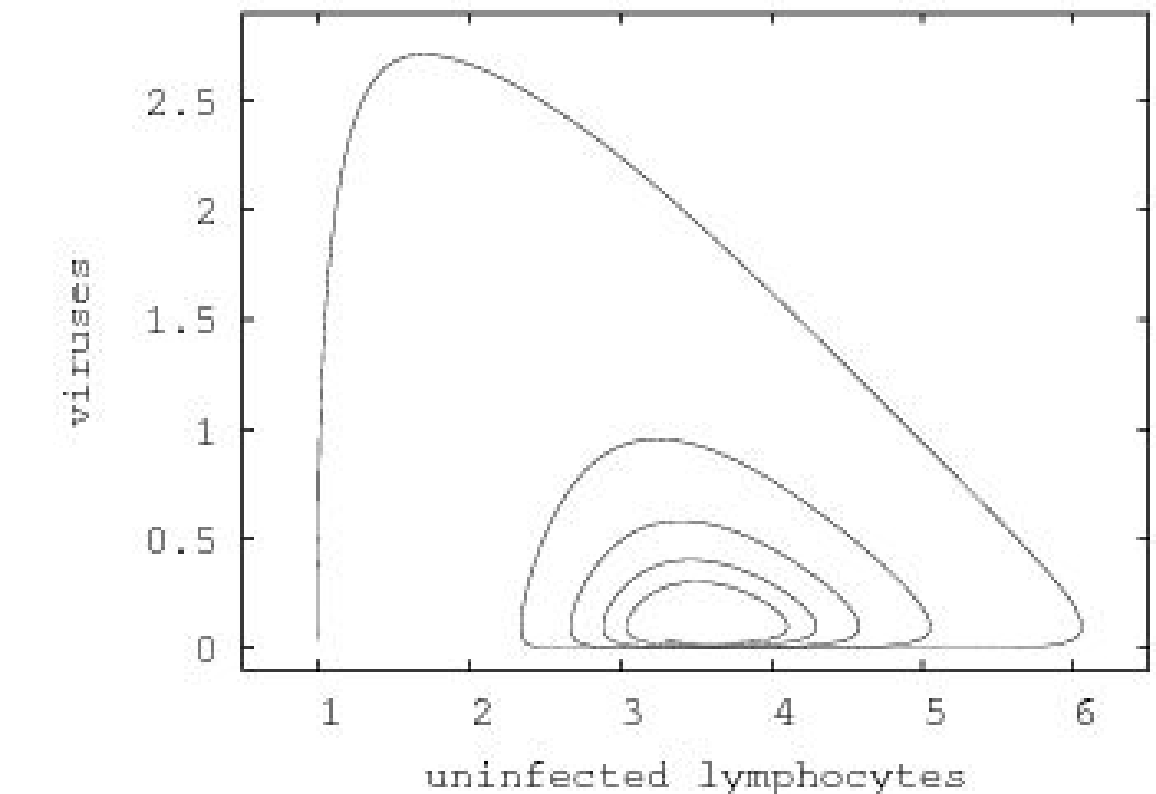}}} 
\hangindent=7mm \hangafter=1 \noindent {\bf Figure 8}. $\,$   
{\it  Numerical results for our model.
Case where a seropositive state is stable ($ \tau = 10 ,\, \zeta=1  $). 
Relatively short time evolution on the left 
(25 days) and phase plane for
lymphocytes and viruses on the right for a 600 days evolution.   }
\bigskip            

 \centerline {
{\includegraphics[width=.49 \textwidth]{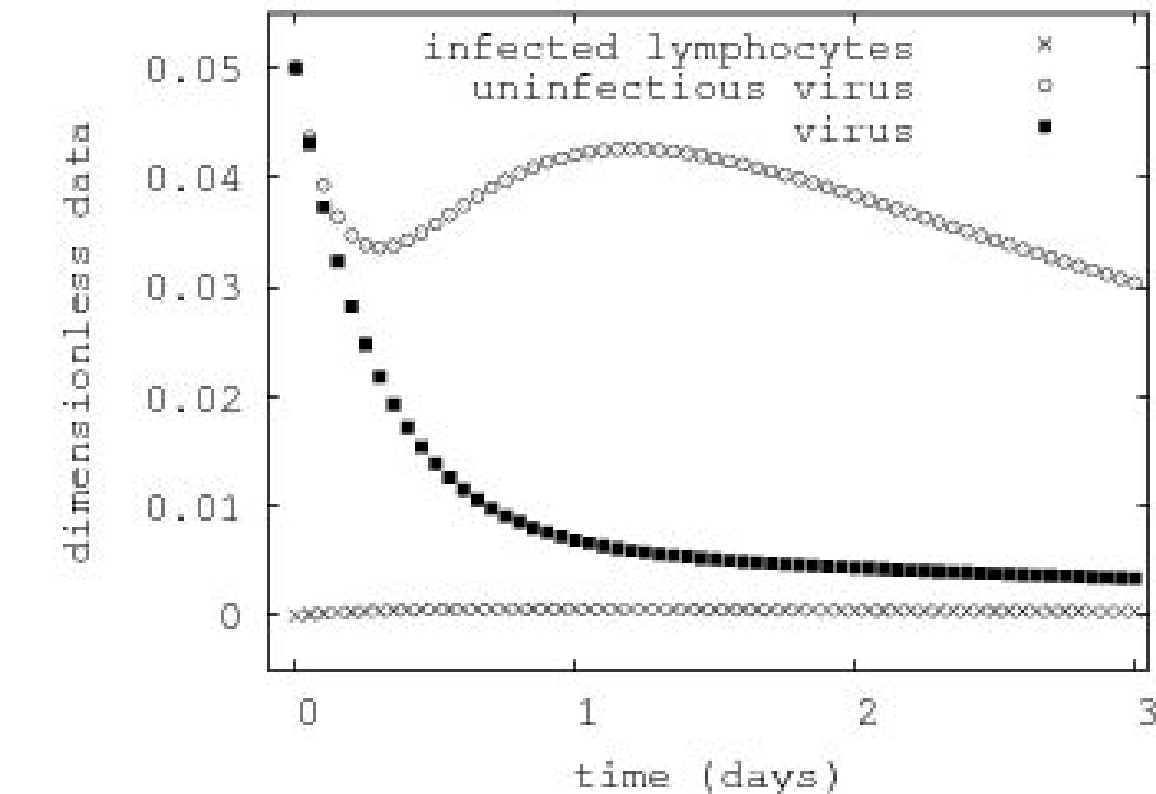}}  
{\includegraphics[width=.49 \textwidth]{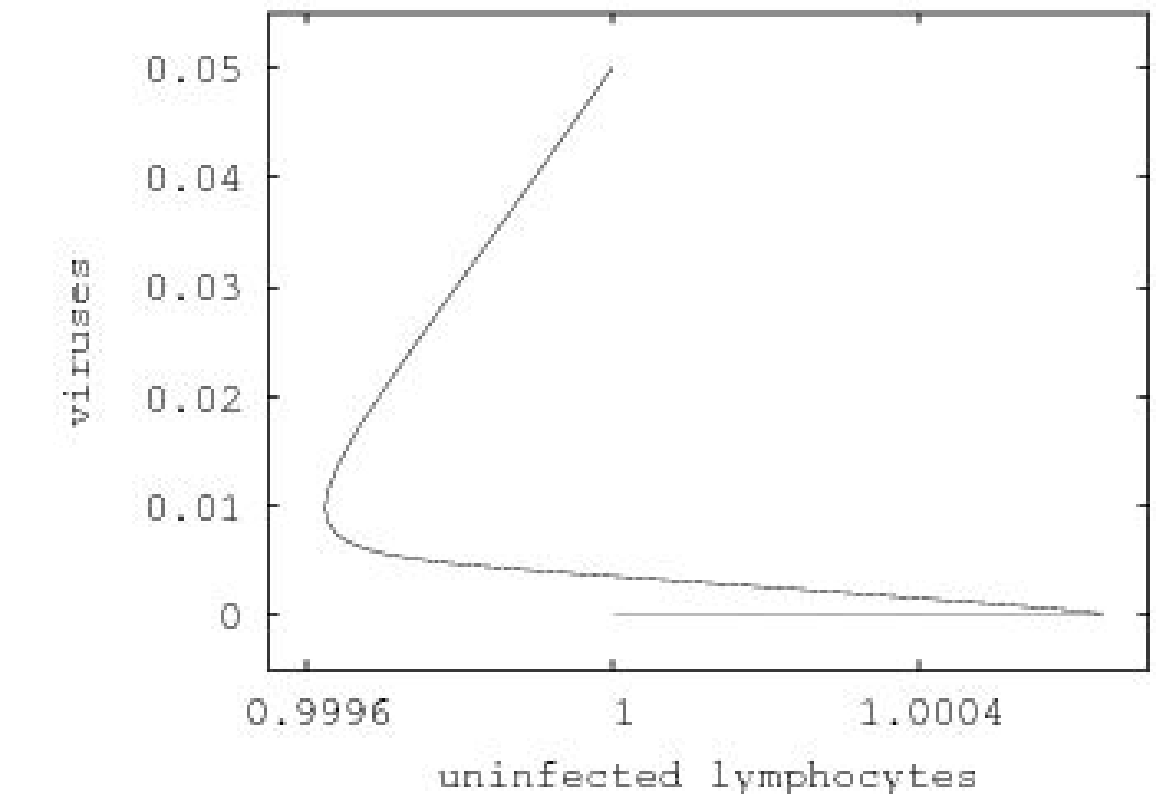}}}
\hangindent=7mm \hangafter=1 \noindent {\bf Figure 9}. $\,$   
{\it Numerical results for our model.
Case where the health is stable  ($ \tau = 20 ,\br   \zeta=3  $).  
Short time evolution on the left (3 days) and phase plane for
lymphocytes and viruses on the right for a 600 days evolution.  }
\bigskip             

  \centerline {
{\includegraphics[width=.49 \textwidth]{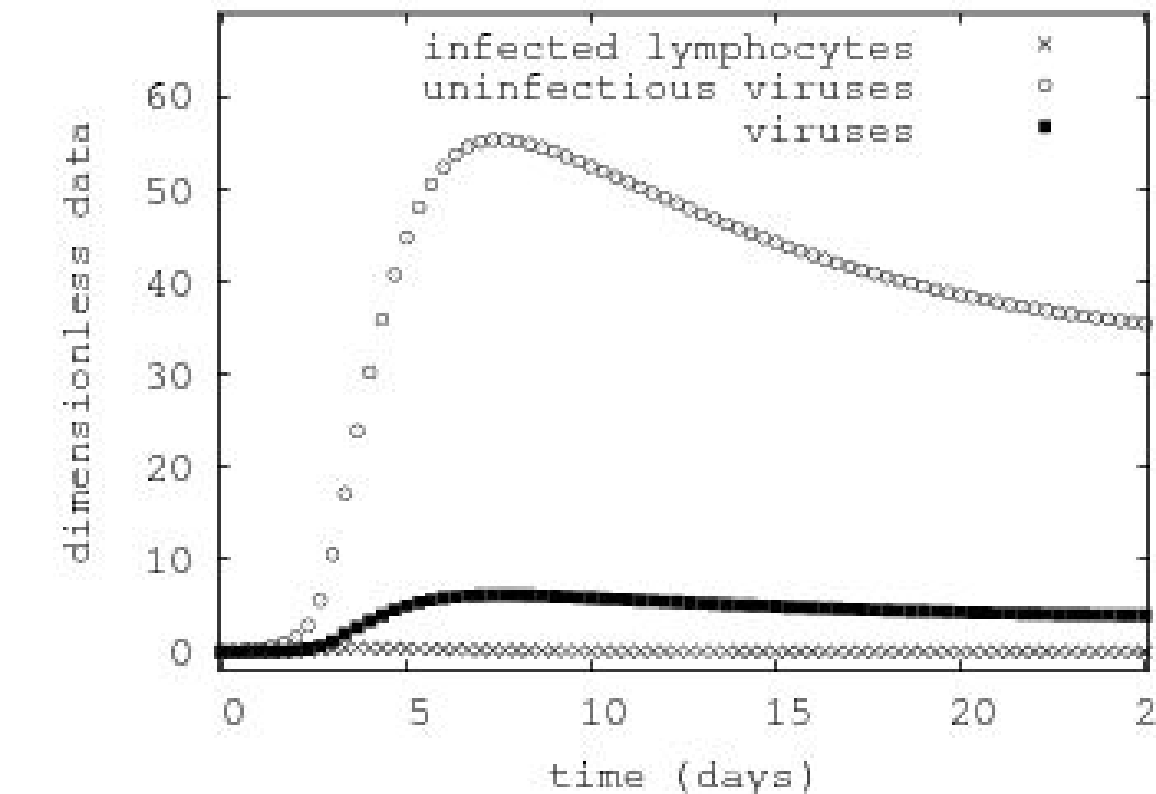}} 
{\includegraphics[width=.49 \textwidth]{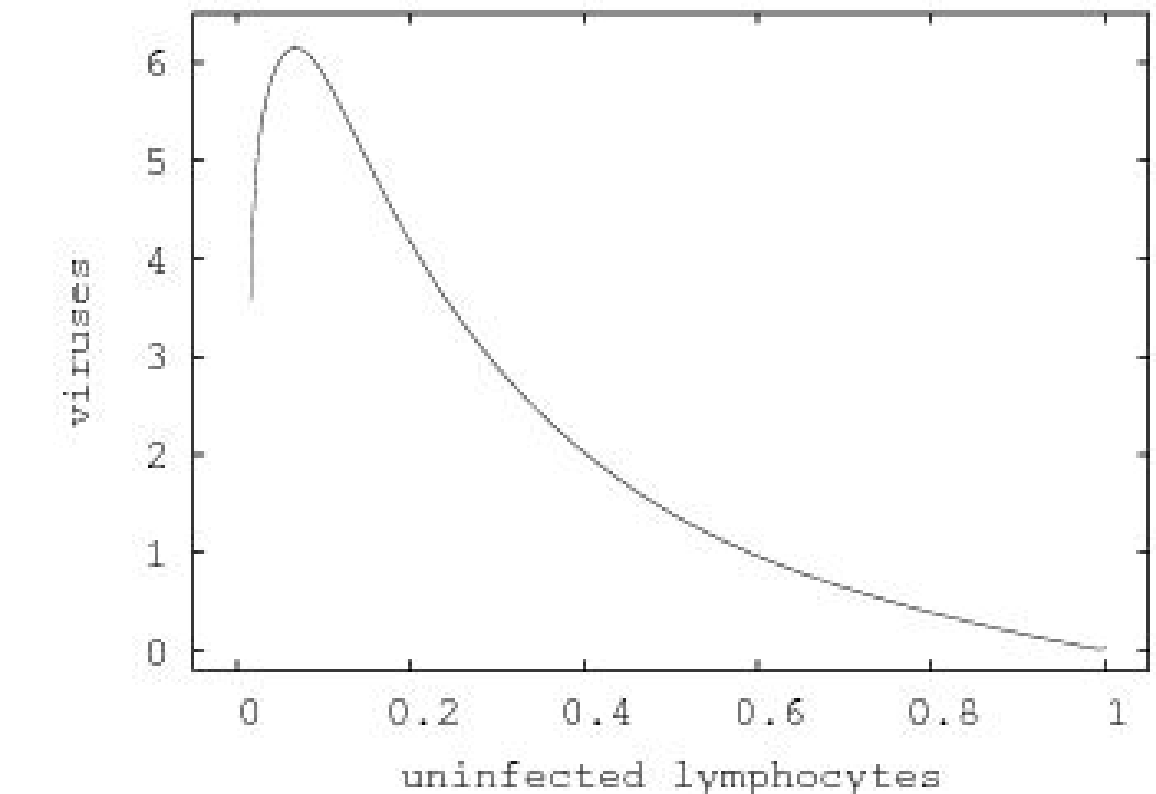}}}
\hangindent=7mm \hangafter=1 \noindent {\bf Figure 10}. $\,$   
{\it Numerical results for our model.
Case where a seropositive state is stable ($ \tau = 1 ,\, \zeta=10  $).  
Relatively short time evolution on the left 
(25 days) and phase plane for
lymphocytes and viruses on the right for a 600 days evolution. 
The effect of the infection is to reduce drastically the efficacy of the 
immune system. }
\bigskip                 
 
\newpage 
\centerline { 
{\includegraphics[width=.49 \textwidth]{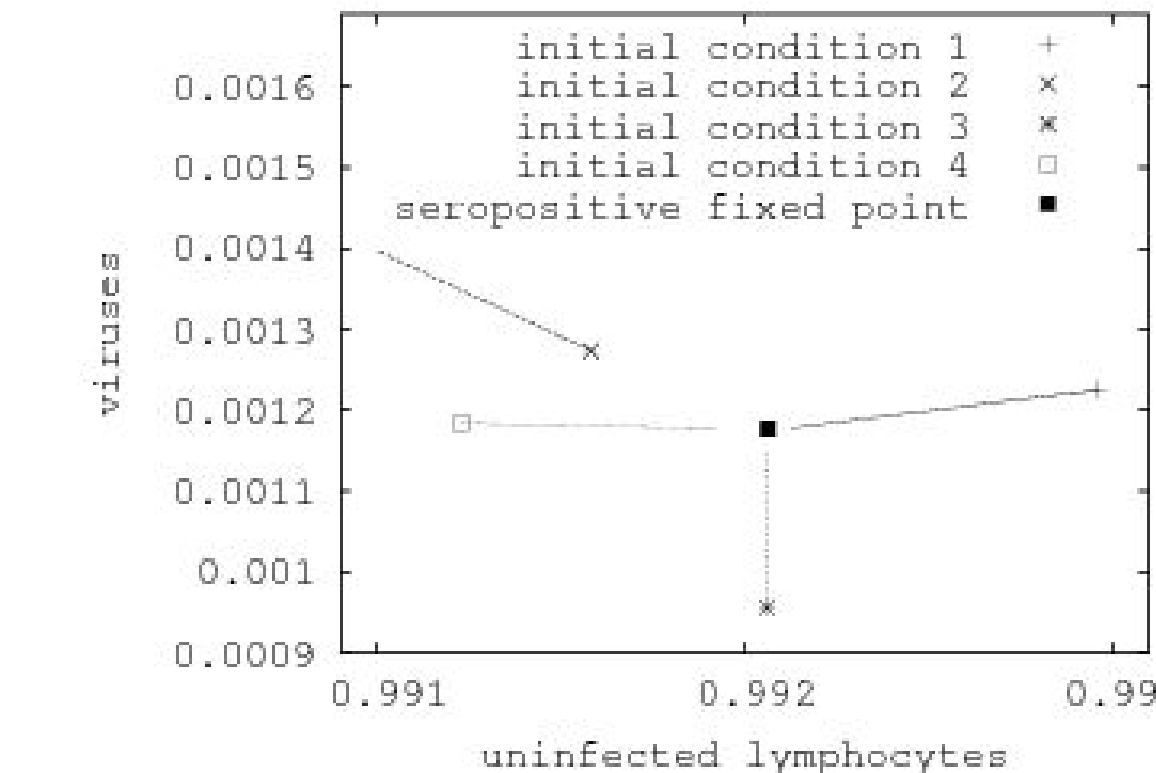}}} 
\hangindent=7mm \hangafter=1 \noindent {\bf Figure 11}. $\,$   
{\it Numerical results for our model with 4 equations ($ \tau = 6,\, \zeta=6 $). 
Simulations are initialized with states very close to the fixed point in directions that
correspond to eigenvectors.
The seropositive state is unstable in one direction among four. }
\bigskip  \bigskip              

 \centerline {
{\includegraphics[width=.49 \textwidth]{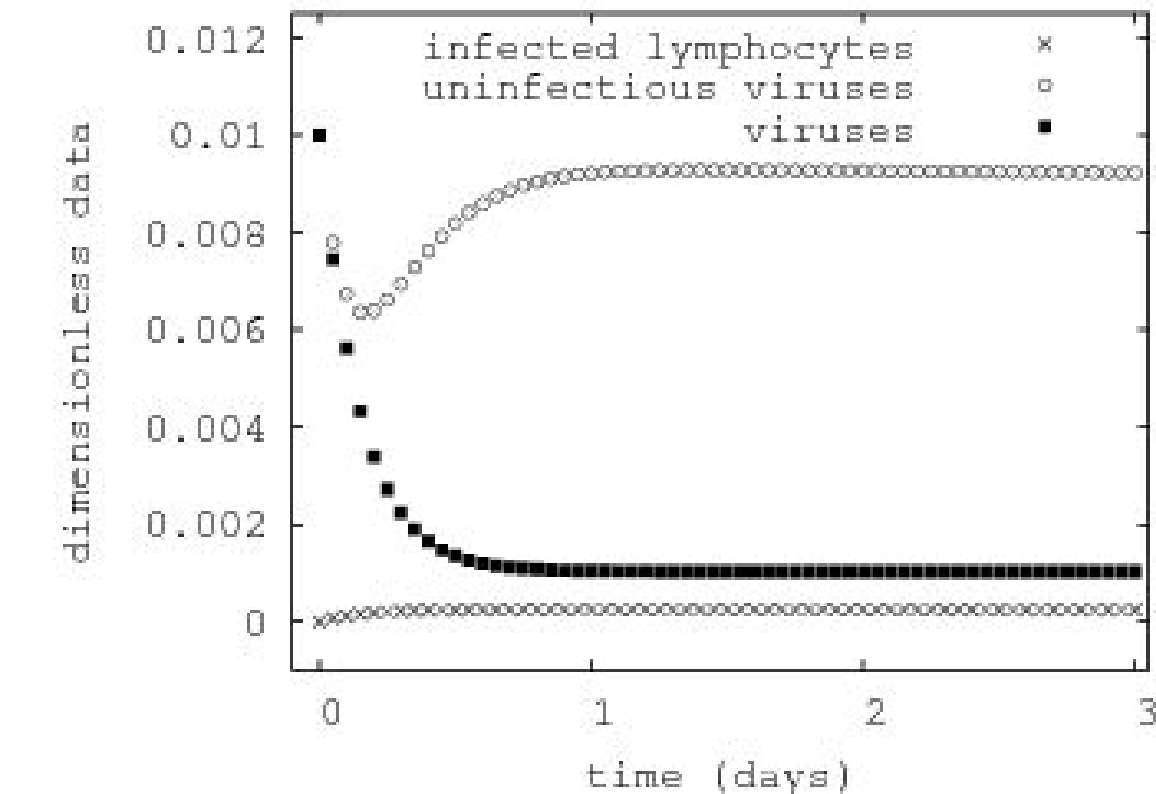}} 
{\includegraphics[width=.49 \textwidth]{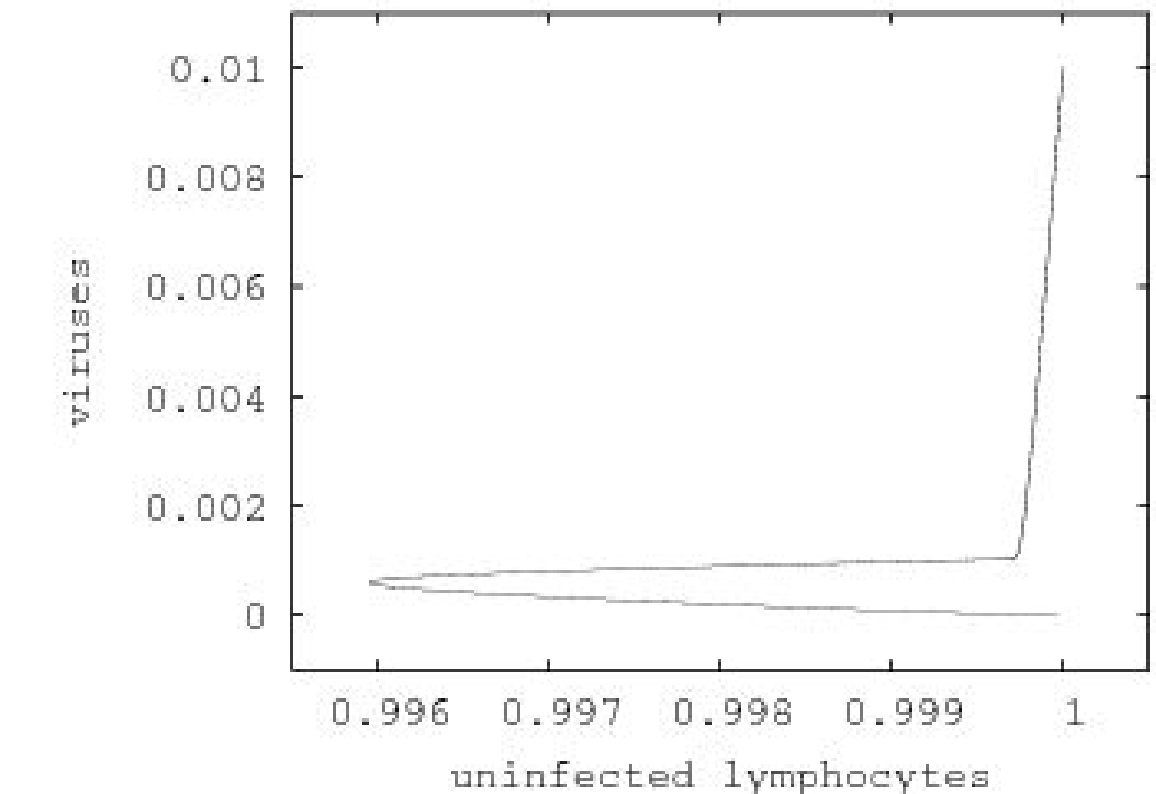}}}
\hangindent=7mm \hangafter=1 \noindent {\bf Figure 12}. $\,$   
{\it  Numerical results for our model with 4 equations.
Case where health is stable ($ \tau = 6,\, \zeta=6  $)
only for small perturbations.  Relatively short time evolution on the left 
(3 days) and phase plane for
lymphocytes and viruses on the right for a 1500 days evolution. }
\bigskip   \bigskip              

 \centerline {
{\includegraphics[width=.49 \textwidth]{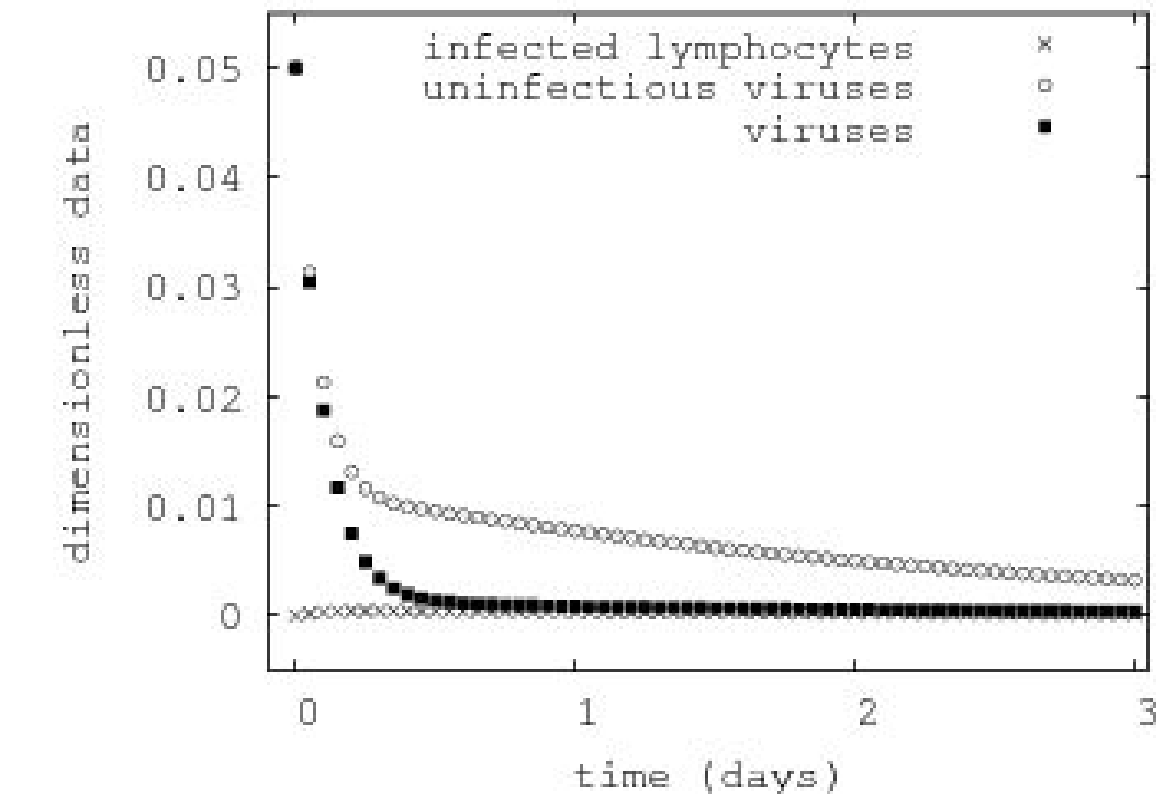}}   
{\includegraphics[width=.49 \textwidth]{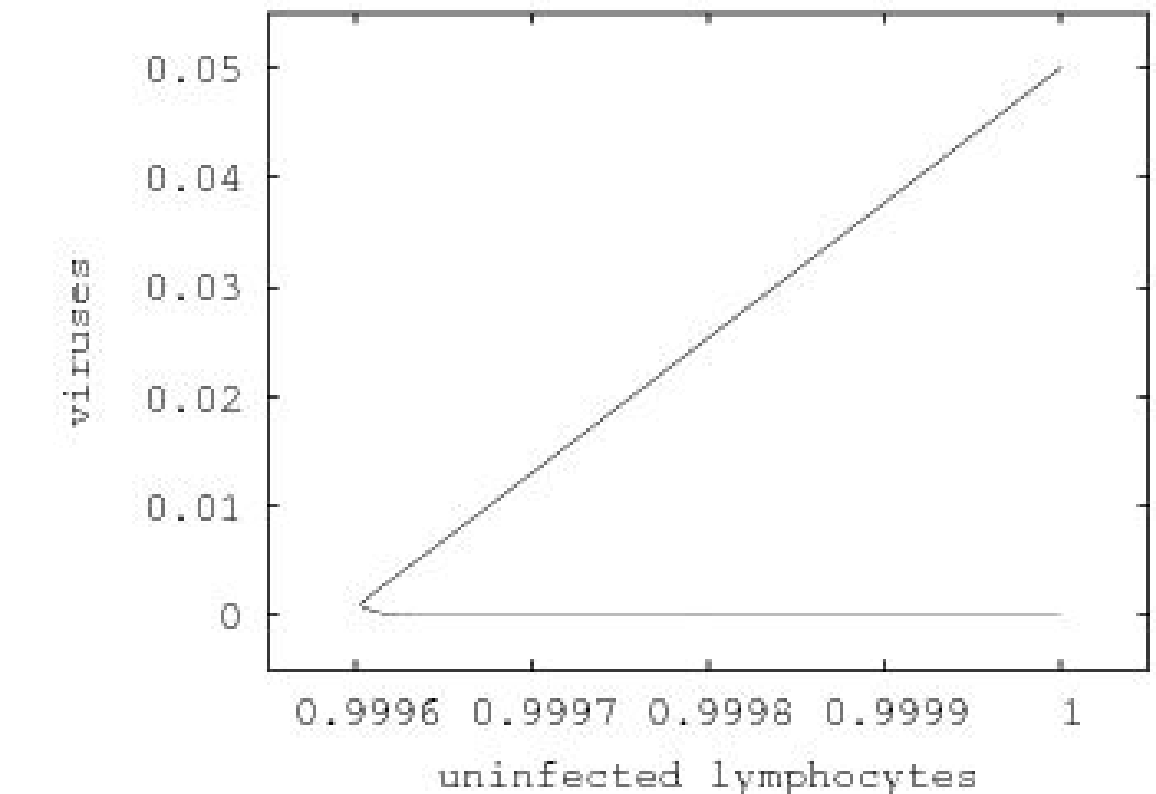}}}
\hangindent=7mm \hangafter=1 \noindent {\bf Figure 13}. $\,$   
{\it  Numerical results for our model.
Case where the health is stable  ($ \tau = 10 , \br \, \zeta=10  $).  
Short time evolution on the left (3 days) and phase plane for
lymphocytes and viruses on the right for a 600 days evolution. }
\bigskip   \bigskip              

\newpage 
\noindent 2. \quad   $\eta <0$.  

\quad {\em (i)} \quad If $\rho <1$: health (partially unstable) and one
  seropositivity numerically stable as can be seen on Figure~10  
($\eta=-4.5, \rho=0.28$). The effect of the
infection is to {\em reduce} the activity of the 
immune system at a very low level ($\simeq 0.015 $ times the 
level of  health);

\quad {\em (ii)} \quad If $1 < \rho < L$: health (stable) and two
  seropositivities. We choosed the parameters $\zeta = 6, \tau=6$ 
($\eta=-1.7,$    
  $\rho=1.008$) for simulation. A first seropositive state ($T^* \simeq 0.129, \; U^*\simeq
  0.03073, \; V^*\simeq 0.992, \; W^*\simeq  8.9$) was numerically found to be locally
  stable (not shown). A second seropositive state ($T^* \simeq 0.992,\; U^*\simeq 0.00028,
  \; V^*\simeq 0.00117, \; W^*\simeq  0.01058$) is very close to health and is locally
  unstable since it has three negative and one positive eigenvalues. So as to illustrate
  this, we have taken initial data
\moneqstar   
X_j=X^* +\varepsilon V_j, \; \; 1 \leq j \leq 4,
\monendstar 
where $X^*$ is the fixed point, $V_j$ is one of its eigenvector and
$\varepsilon$ is sufficiently small. This experiment is depicted in Figure~11.
Note that although health is locally stable, initial conditions
with a viral load of 5\% drove the state to a seropositivity state. This proves that the
basin of attraction is small. So we provide a simulation with an initial viral load of
only 1\% in Figure~12~; 

\quad {\em (iii)} \quad If $L < \rho $: health is stable as one may see on 
Figure~13 
($\eta=-4.5, \rho=2.8$). Minimum value for
lymphocytes is obtained for $t \simeq 2.6 $ days.

\bigskip \bigskip   \noindent {\bf \Large 5) \quad  Discussion of the present model}
  
\smallskip 
Various effects are supposed to be more or less incorporated in any model and specifically
ours. They are discussed hereafter.

\monitem If a model considers the field of {\em free} viruses, then
infection makes an uninfected lymphocyte and a {\em free} virus
disappear and an infected lymphocyte appear at the same time. Three
identical terms (up to a $\pm$ sign) should be present in such a
model. It is not often the case. Since $V_j$ includes not only
free viruses but also {\em infecting} viruses, we are not forced to
have the same three terms.

\monitem We characterize each virus by its antigenicity and by the
information that it is infectious or not. We characterize the
lymphocytes by the virus antigenicity against which they have been
designed. Mutation is then only a probabilistic phenomenon and the
main modeling question is the space in which it takes place and its
probabilistic law. Such a study is postponed to a forthcoming article.

\monitem By explicitly deriving our model, we justify our ``piecewise
linear'' term to model infection, although most authors use
a mass-action quadratic term (for a discussion, see \cite{Callaway_Perelson,Boer}).

\noindent 
Very likely, the reason why the modeling of infection has not
been much studied is that the only biologically measured field is
$T+U$. So the term modeling this phenomenon disappears in any
evolution equation on $T+U$. Yet it models a crucial
reality and deserves more attention.

\monitem We model the production of $T_j$ by the immune system once it
has detected the $V_j$. The production depends on the $V_j$
population, and so our term is quadratic. It has no counterpart in any
model we read. This term could be balanced by the quadratic term
modeling the immune system effect against each strain of virus which
is present in some models (\cite{Frid_Jabin_Perthame,Pastore_Zubelli})
including ours.

\monitem The multiplication/mutation of the virus is the most
challenging phenomenon. Since only one virus may infect a lymphocyte
\cite{Dern}, we may assume $U$ is a good measure of the total number
of {\em infecting} viruses. This assumption is very likely but would
deserve to be further tested. Then, since the antigenicity of a $U_j$
has no link with the antigenicity of the virus infecting it, we need
to take into account the pure multiplication of $V_j$ phenomenon. A
simple modeling of mutation generates our term.

\monitem We model the effect of the immune system against the viruses by a
term depending quadratically on $V_j$ and $T_j$ since these quantities are
effective in the same regime. Such a term can be found in the models
depending on antigenic variation in \cite{Nowak-May} (chapter 12 and
13) and \cite{Pastore_Zubelli} but in these models the lymphocyte's
generation is modeled only through a linear term in $V_j$.

\monitem With further assumptions, one may find a linear combination of the
$T+U$ and $V+W$ evolution such that this new combination is simply
linear in $T,U,V,W$ because the non-linearities may simplify. This
could be experimentally tested.

\monitem The overall behavior of all the systems studied above (including ours if $N=1$)
allows the fields to remain non-negative and be attracted by some fixed points.
So all these models predict convergence to some fixed point which is
never immuno-suppresed ($T=0$). We consider this to be a major drawback for
the long term modeling of HIV infection. This opinion is shared by the
authors of \cite{Pastore_Zubelli} and they propose modifications to
the Nowak-Bangham models enabling the longer term evolution
modeling.

\monitem Notice that the only physical field is $T= \sum_j T_j$ and not $T_1$
(with $N=1$) as used in our mathematical study. As a consequence, 
the widening of antigenicity support
is a phenomenon not included in the case $N=1$ nor in any other
well-known ``macroscopic'' model reviewed above. Since the
``microscopic'' models use a finite domain of antigenicity, they
no more include the widening of antigenicity support. As 
de Boer and Perelson conclude their study of numerous macroscopic models
in \cite{Boer_Perelson}: ``one may model
disease progression by allowing the virus to evolve immune-escape
variants increasing the diversity of the quasi-species [...]. Since this requires
high-dimensional models, this form of
disease progression is not considered any further here''.
It would of course be of interest to get an experimental
illustration of the viral dynamics as exposed in the present study. As
described in \cite{Derrien}, nucleotide analogs can be selected to
take control over the mutational drift of the virus, in particular abolishing ($N=1$), or
reducing (one may even reach precisely $N=2$) the
drift. Starting with a $N=1$ experiment (no drift), $T$, $U$ can be
counted by flow cytometry, for instance using a dye adsorbed by life
cells ($T_1$), not by dead ones ($U_1$). The virus populations $V_1$ and $W_1$ can be counted by
the capacity to infect ($V_1$) or not infect ($W_1$) $T_1$ cells. The latter
will be labeled (for instance fluorescence-tagged antibody grown
against $T_1$).  The same counting would be repeated in a $N=2$
experiment, in which the fluorescent label of $T_1$ will allow to count
$T_2$, etc... An entirely different strategy would be to look for a
macroscopic version of our model. Such a model would depend on $T, U,
V, W$ instead of $(T_i, U_i, V_i, W_i)_{i=1, ..., N}$, but
remains to be determined.

\bigskip  \smallskip  \noindent {\bf \Large 6) \quad Conclusion }   
    
\smallskip
We have thoroughly studied previous models of HIV multiplication by systems of
differential equations. Some of them were reduced to
be single-antigenic. With such a reduction, all of these models have fixed
points that prohibit modeling of the last phase of the disease where
the T count vanishes. This is also criticized in recent research \cite{Pastore_Zubelli}.

Moreover, we propose a model taking into
account new phenomena among which lymphocytes generation by the immune
system according to the presence of specific viruses, and immune
effect against each virus strain. We also model infection and
mutation/generation through new algebraic terms. This model is derived
due to explicit arguments. It will be tested further in forthcoming
research.

Although the {\em reduced} version of our model has the same drawback of not
enabling the immunity exhaustion, its general version takes into
account the strains' diversity (here denoted as antigenicity) and the
specificity of the immune response. So our full model should enable to
account for the last phase of the HIV infection where the lymphocytes'
count vanishes. This will be studied in a forthcoming article.

\bigskip    \noindent {\bf \large Acknowledgment}

The authors want to thank Laurent Desvillettes and Beno\^{\i}t Perthame for a fruitful
discussion in June 2008 and February 2009 respectively.

\bigskip    
\noindent {\bf \large  References } 

\vspace{-.3cm} 
\bibliographystyle{plain}
\bibliography{biblio}

\begin{thebibliography}{10}

\bibitem{Althaus_Boer}
Christian~L. Althaus and Rob~J. De~Boer.
\newblock Dynamics of immune escape during {HIV/SIV} infection.
\newblock {\em PLoS Comput. Biol.}, 4(7):e1000103, 07 2008.

\bibitem{Anderson}
JP. Anderson, R.~Daifuku, and LA. Loeb.
\newblock Viral error catastrophe by mutagenic nucleosides.
\newblock {\em Annu. Rev. Microbiol.}, 58:183--205, 2004.

\bibitem{Arnold}
VI. Arnold.
\newblock {\em Ordinary Differential Equations}.
\newblock The MIT Press., Massachusetts, 1978.

\bibitem{Bebenek}
K.~Bebenek, J.~Abbotts, SH. Wilson, and TA. Kunkel.
\newblock Error-prone polymerization by {HIV}-1 reverse transcriptase.
  {C}ontribution of template-primer misalignment, miscoding, and termination
  probability to mutational hot spots.
\newblock {\em J. Biol. Chem.}, 268(14):10324--34, 1993.

\bibitem{Biebricher}
CK. Biebricher and M.~Eigen.
\newblock The error threshold.
\newblock {\em Virus Res.}, 107(2):117--127, 1993.

\bibitem{Bortz_Nelson}
D.M. Bortz and P.~W. Nelson.
\newblock Model selection and mixed-effects modeling of {HIV} infection
  dynamics.
\newblock {\em Bull. Math. Biol.}, 68(8):2005--2025, 2006.

\bibitem{Brass}
AL. Brass, DM. Dykxhoorn, Y.~Benita, N.~Yan, A.~Engelman, RJ. Xavier,
  J.~Lieberman, and Elledge SJ.
\newblock Identification of host proteins required for {HIV} infection through
  a functional genomic screen.
\newblock {\em Science}, 319(5865):921--6, 2008.

\bibitem{Callaway_Perelson}
DS. Callaway and AS. Perelson.
\newblock {HIV}-1 infection and low steady state viral loads.
\newblock {\em Bull. Math. Biol.}, 64:29--64, 2002.

\bibitem{Boer}
RJ. De~Boer.
\newblock Understanding the failure of {CD}8+ t-cell vaccination against
  simian/human immunodeficiency virus.
\newblock {\em J. Virol.}, 81(6):2838--2848, 2007.

\bibitem{Boer_Perelson}
RJ. De~Boer and AS. Perelson.
\newblock Target cell limited and immune control models of {HIV} infection: A
  comparison.
\newblock {\em J. Theor. Biol.}, 190:201--214, 1998.

\bibitem{Leenheer_Smith}
P.~De~Leenheer and HL. Smith.
\newblock Virus dynamics: a global analysis.
\newblock {\em SIAM J. Appl. Math.}, 63(4):1313--1327, 2003.

\bibitem{Dern}
K.~Dern, H.~R{\"{u}}bsamen-Waigmann, and RE. Unger.
\newblock Inhibition of {HIV} type 1 replication by simultaneous infection of
  peripheral blood lymphocytes with human immunodeficiency virus types 1 and 2.
\newblock {\em AIDS Res. Hum. Retroviruses}, 17(4):295--309, 2001.

\bibitem{Derrien}
V.~Derrien.
\newblock {\em Fidelity and termination of polymerization by reverse
  transcriptases {\em in vitro}.}
\newblock PhD thesis, Paris-Sud University, 1998.

\bibitem{Drosopoulos}
WC. Drosopoulos, LF. Rezende, MA. Wainberg, and VR. Prasad.
\newblock Virtues of being faithful: can we limit the genetic variation in
  human immunodeficiency virus?
\newblock {\em J. Mol. Med.}, 76(9):604--12, 1998.

\bibitem{Dubois07}
F.~Dubois, HVJ. Le~Meur, and C.~Reiss.
\newblock Mod{\'e}lisation de la multiplication du virus {HIV}.
\newblock In Congr{\`e}s d'Analyse NUM{\'e}rique, editor, {\em SMAI},
  Praz-sur-Arly, 2007.

\bibitem{Finzi}
D.~Finzi and RF. Silliciano.
\newblock Viral dynamics in {HIV}-1 infection.
\newblock {\em Cell}, 93(5):665--71, 1998.

\bibitem{Frid_Jabin_Perthame}
H.~Frid, P-E. Jabin, and B.~Perthame.
\newblock Global stability of steady solutions for a model in virus dynamics.
\newblock {\em Math. Model. Numer. Anal.}, 37(4):709--723, 2003.

\bibitem{Fujii}
K.~Fujii, JH. Hurley, and EO. Freed.
\newblock Beyond {Tsg101}: the role of {Alix} in {'ESCRTing' HIV}-1.
\newblock {\em Nat. Rev. Microbiol.}, 5(12):912--6, 2007.

\bibitem{Guedj_Thiebaut_Commenges}
J.~Guedj, R.~Thiebaut, and D.~Commenges.
\newblock Practical identifiability of {HIV} dynamics models.
\newblock {\em Bull. Math. Biol.}, 69(8):2493--513, 2007.

\bibitem{Harris}
KS. Harris, W.~Brabant, S.~Styrchak, A.~Gall, and R.~Daifuku.
\newblock {KP-1212/1461}, a nucleoside designed for the treatment of {HIV} by
  viral mutagenesis.
\newblock {\em Antivir. Res}, 67:1--9, 2005.

\bibitem{Henrici}
P.~Henrici.
\newblock {\em Applied and Computational Complex Analysis}.
\newblock John Wiley, New-York, 1974.

\bibitem{Herz}
AV. Herz, S.~Bonhoeffer, RM. Anderson, RM. May, and MA. Nowak.
\newblock Viral dynamics {\em in vivo}: limitations on estimates of
  intracellular delay and virus decay.
\newblock {\em Proc. Natl. Acad. Sci. USA}, 93:7247--7251, 1996.

\bibitem{Ji}
J.~Ji and LA. Loeb.
\newblock Fidelity of {HIV}-1 reverse transcriptase copying a hypervariable
  region of the {HIV}-1 env gene.
\newblock {\em Virology}, 199(2):323--30, 1994.

\bibitem{Kashkina}
E.~Kashkina, M.~Anikin, F.~Brueckner, RT. Pomerantz, WT. McAllister, P.~Cramer,
  and D.~Temiakov.
\newblock Template misalignment in multisubunit {RNA} polymerases and
  transcription fidelity.
\newblock {\em Mol. Cell.}, 24(2):257--66, 2006.

\bibitem{Kothe}
DL. Kothe, Y.~Li, JM. Decker, F.~Bibollet-Ruche, KP. Zammit, MG. Salazar,
  Y.~Chen, Z.~Weng, EA. Weaver, F.~Gao, BF. Haynes, GM. Shaw, BT. Korber, and
  BH. Hahn.
\newblock Ancestral and consensus envelope immunogens for {HIV}-1 subtype {C}.
\newblock {\em Virology}, 352(2):438--49, 2006.

\bibitem{Mueller}
S.~Mueller, D.~Papamichail, JR. Coleman, S.~Skiena, and E.~Wimmer.
\newblock Reduction of the rate of poliovirus protein synthesis through
  large-scale codon deoptimization causes attenuation of viral virulence by
  lowering specific infectivity.
\newblock {\em J. Virol.}, 80(19):9687--96, 2006.

\bibitem{Murakami}
E.~Murakami, A.~Basavapathruni, WD. Bradley, and KS. Anderson.
\newblock Mechanism of action of a novel viral mutagenic covert nucleotide:
  molecular interaction with {HIV}-1 reverse transcriptase and host cell {DNA}
  polymerases.
\newblock {\em Antivir. Res.}, 67:10--17, 2005.

\bibitem{Nelson}
PW. Nelson, JE. Mittler, and AS. Perelson.
\newblock Effect of drug efficacy and the eclipse phase of the viral life cycle
  on estimates of {HIV} viral dynamic parameters.
\newblock {\em J. Acquir. Immune Defic. Syndr.}, 26(5):405--412, 2001.

\bibitem{Nowak_Bangham}
MA~Nowak and Bangham. CRM.
\newblock Population dynamics of immune responses to persistent viruses.
\newblock {\em Science}, 272:74--79, 1996.

\bibitem{Nowak-May}
MA. Nowak and RM. May.
\newblock {\em Virus dynamics. Mathematical Principles of Immunology and
  Virology}.
\newblock Oxford University Press, 2000.

\bibitem{Pastore_Zubelli}
DH. Pastore and JP. Zubelli.
\newblock On the dynamics of certain models describing the {HIV} infection.
\newblock In Mauricio~Peixoto Alberto Adrego~Pinto and David Rand, editors,
  {\em Dynamics and Games in Science, in honour of Mauricio Peixoto and David
  Rand}. Springer-Verlag, 2010.

\bibitem{perelson02}
AS. Perelson.
\newblock Modelling viral and immune system dynamics.
\newblock {\em Nature Rev. Immunol.}, 2:28--36, 2002.

\bibitem{perelson99mathematical}
AS. Perelson and P.~Nelson.
\newblock Mathematical analysis of {HIV}-1 dynamics {\em in vivo}.
\newblock {\em SIAM Review}, 41(1):3--44, 1999.

\bibitem{Perelson}
AS. Perelson, AU. Neumann, M.~Markowitz, JM. Leonard, and DD. Ho.
\newblock {HIV}-1 dynamics {\em in vivo}: virion clearance rate, infected cell
  life-span, and viral generation time.
\newblock {\em Science}, 271(5255):1582--6, 1996.

\bibitem{Petravic}
J.~Petravic, L.~Loh, SJ. Kent, and MP. Davenport.
\newblock {CD4+} target cell availability determines the dynamics of immune
  escape and reversion {\em in vivo}.
\newblock {\em J. Virol.}, 82(8):4091--101, 2008.

\bibitem{Recher}
M.~Recher, KS. Lang, A.~Navarini, L.~Hunziker, PA. Lang, K.~Fink, S.~Freigang,
  P.~Georgiev, L.~Hangartner, R.~Zellweger, A.~Bergthaler, AN. Hegazy,
  B.~Eschli, A.~Theocharides, LT. Jeker, D.~Merkler, B.~Odermatt,
  M.~Hersberger, H.~Hengartner, and RM. Zinkernagel.
\newblock Extralymphatic virus sanctuaries as a consequence of potent {T}-cell
  activation.
\newblock {\em Nat. Med.}, 13(11):1316--23, 2007.

\bibitem{Sachsenberg}
N.~Sachsenberg, AS. Perelson, S.~Yerly, GA. Schockmel, D.~Leduc, B.~Hirschel,
  and L.~Perrin.
\newblock Turnover of {CD4$^+$} and {CD8$^+$} lymphocytes in {HIV-1} infection
  as measured by {Ki-67} antigen.
\newblock {\em J. Exp. Medecine}, 187(8):1295--1303, 1998.

\bibitem{Samson_Lavielle_Mentre}
A.~Samson, M.~Lavielle, and F.~Mentr{\'e}.
\newblock The {SAEM} algorithm for group comparison tests in longitudinal data
  analysis based on non{-}linear mixed{-}effects model.
\newblock {\em Stat Med.}, 26(27):4860--75, 2007.

\bibitem{Snedecor03}
SJ. Snedecor.
\newblock Comparison of three kinetic models of {HIV-1} implications for
  optimization of treatment.
\newblock {\em J. Theor. Biol.}, 221:519--541, 2003.

\bibitem{Thomas}
JA. Thomas, DE. Ott, and RJ. Gorelick.
\newblock Efficiency of human immunodeficiency virus type 1 postentry infection
  processes: evidence against disproportionate numbers of defective virions.
\newblock {\em J. Virol.}, 81(8):4367--70, 2007.

\bibitem{Tuckwell}
HC. Tuckwell and E.~Le~Corfec.
\newblock A stochastic model for early {HIV-1} population dynamics.
\newblock {\em J. Theor. Biol.}, 195(4):451--63, 1998.

\bibitem{Wang_Li}
L.~Wang and M.Y. Li.
\newblock Mathematical analysis of the global dynamics of a model for {HIV}
  infection of {$\rm CD4\sp +$ T} cells.
\newblock {\em Math. Biosci.}, 200(1):44--57, 2006.

\bibitem{Warrilow}
D.~Warrilow, L.~Meredith, A.~Davis, C.~Burrell, P.~Li, and D.~Harrich.
\newblock Cell factors stimulate human immunodeficiency virus type 1 reverse
  transcription {\em in vitro}.
\newblock {\em J. Virol.}, 82(3):1425--37, 2008.

\bibitem{Wodarz_Hamer}
D.~Wodarz and D.H. Hamer.
\newblock Infection dynamics in {HIV}-specific {CD4 T} cells: does a {CD4 T}
  cell boost benefit the host or the virus?
\newblock {\em Math. Biosci.}, 209(1):14--29, 2007.

\end{thebibliography}

\end{document}